\documentclass{amsart}
\usepackage{amssymb,amsmath}
\usepackage{graphicx}
\usepackage{subfigure}
\usepackage{upgreek}
\usepackage{tikz}
\usepackage{hyperref}
\usepackage{mathtools}
\usetikzlibrary{matrix}
\usepackage{cmap}
\usepackage[T1]{fontenc}

\usepackage[all]{xy}
\usepackage{optparams,eurosym}
\usepackage{amsthm}

\numberwithin{equation}{section}

\newcommand\paperbody%
        {}

\newtheorem{lemma}{Lemma}[section]
\newtheorem{proposition}[lemma]{Proposition}

\newtheorem{theorem}[lemma]{Theorem}
\newtheorem{non-theorem}{Non-Theorem}
\newtheorem{conjecture}{Conjecture}

\newtheorem{definition}{Definition}
\newtheorem{remark}[lemma]{Remark}

\newtheorem{example}{Example}

\newcommand\abs[1]{|#1|}

\numberwithin{equation}{section}
\begin{document}
\begin{abstract}
In this paper, we confirm a conjecture of Orlik-Randell from 1977 on the Seifert form of chain type invertible singularities. We use Lefschetz bifibration techniques as developed by Seidel (inspired by Arnold and Donaldson) and take advantage of the symmetries at hand. We believe that our method will be useful in understanding the homological/categorical version of Berglund-H\"ubsch mirror conjecture for invertible singularities. 
\end{abstract}
\title[Seifert form of chain type invertible singularities ]{Seifert form of chain type invertible singularities }
\author{Umut Varolgunes}

\maketitle

\section{Introduction}\label{s1}

For an $n$-tuple of positive integers $a=(a_1,..,a_n)\in \mathbb{Z}^n_{>0}$, $n\geq 1$, we define the polynomial:
\begin{align}p_a(z_1,\ldots,z_n):= \sum_{i=1}^{n-1}z_i^{a_i}z_{i+1}+z_n^{a_n}.
\end{align}
These are called the chain type invertible polynomials. They are quasi-homogeneous polynomials, and if $a_n\geq 2$, they have an isolated singularity at the origin when considered as a map $\mathbb{C}^n \to \mathbb{C}$. We define the group of symmetries of $p_a$ to be 
\begin{align}\Gamma_a:=\{(\lambda_1,\ldots, \lambda_n,\lambda)\mid \lambda_1^{a_1}\lambda_2=\ldots= \lambda_{n-1}^{a_{n-1}}\lambda_n=\lambda_{n}^{a_{n}}=\lambda\}\subset (\mathbb{C}^*)^{n+1}.\end{align}
There is an inclusion of $\mathbb{C}^*$ into $\Gamma_a$ witnessing the quasi-homogeneity (see Equation \ref{eqquasi} in Section \ref{ss3.1}).
%

The Milnor number and the characteristic polynomial of the singularity of $p_a$ were computed in \cite{MO}. Let us note that the Milnor number is given by 
\begin{align}\mu(a):=a_1\ldots a_n-a_2\ldots a_n+\ldots +(-1)^{n-1}a_n+(-1)^n.\end{align} 
In \cite{OR}, Orlik-Randell computed the integral monodromy of $p_a$ with respect to a carefully chosen basis of the integral homology of the Milnor fiber of $p_a$. They were not able to compute the intersection form, but they conjectured that their basis is in fact a distinguished basis whose Seifert matrix is given by:\begin{equation}\label{seifertmatrix}
S(a):=\begin{bmatrix} 
&1 & \alpha_1 & \alpha_2 & \alpha_3 & \ldots & \alpha_{\mu(a)-1}\\
&0 & 1  & \alpha_1 & \alpha_2 &\ldots & \alpha_{\mu(a)-2}\\
&0 & 0 & 1 &\alpha_1 &\ldots&\alpha_{\mu(a)-3}\\
& &&&\ldots&&&\\
&0 &0&0&\ldots&1&\alpha_1\\
&0 &0&0&0&\ldots&1\\
\end{bmatrix}.
\end{equation}Here  $\alpha_1,\alpha_2,\ldots,\alpha_{\mu(a)-1}\in\mathbb{Z}$ are defined via the equality: \begin{align}\prod_{i=0}^{n}(1-t^{r_i})^{(-1)^{i+n-1}}= 1+ \alpha_1t+ \alpha_2t^2 +\ldots+ a_{\mu(a)-1}t^{\mu(a)-1} (\text{mod } t^{\mu(a)}),\end{align}
inside $\mathbb{Z}[[t]]$, where $r_i := \prod_{j = n-i+1}^{n}a_j$. Our sign conventions are slightly different from the ones of \cite{OR}, and we have stated their conjecture using our conventions (which are explained in Section \ref{ss2.1}). We are also using the inverse of the matrix that appeared in the Orlik-Randell paper (see Remark 3.7 from \cite{AT} and also Lemma \ref{s2lemma}). 

In this paper, we resolve the Orlik-Randell conjecture. Let us be more precise about what we actually prove. In Theorem 2.11 of \cite{OR}, Orlik-Randell give an explicit generators and relations presentation of the reduced integral cohomology $H^{*}$ of the Milnor fiber of $p_a$ as an abelian group. The proof is an inductive argument using the cohomology long exact sequence of a pair (also see Remark \ref{relativeorlik} below). In particular, the generators correspond to a basis of the relative cohomology group constructed in their Theorem 2.10 via \cite{le}, which has special properties in regards to an ad-hoc monodromy action defined in Lemma 2.6 (ii) (nothing else about this basis is specified in \cite{OR}). 

The generators of $H^{*}$ are linearly ordered in their notation but it appears to us that there is only a natural cyclic order. Taking the right number of adjacent generators provide a $\mathbb{Z}$-basis of $H^{*}$ with special properties with respect to the monodromy operator. The only possible way to interpret their conjecture is that the dual basis to any of these bases in reduced integral homology $H_{*}=Hom(H^*,\mathbb{Z})$ of the Milnor fiber is distinguished with the Seifert matrix in Equation \ref{seifertmatrix}. 

\begin{remark}
Notice that $H_*$ is consequently presented as the solution set of some linear equations in a free abelian group (the relative homology group) with an induced basis. Yet, the induced dual bases of $H_*$ as in the previous paragraph are not obtained by taking the corresponding adjacent basis elements of the ambient abelian group - they must be corrected (in a unique way) so that they lie in the solution set.
\end{remark}

Strictly speaking we do not prove or disprove this form of the conjecture. We believe that the following statement captures the actual content of the Orlik-Randell conjecture.

\begin{theorem}[Modified Orlik-Randell conjecture]\label{s1thm} The (reduced) integral homology of the Milnor fiber of $p_a$ admits a cylically ordered collection of $r_n=a_1\ldots a_n$ many elements $v_1,\ldots ,v_{r_n}$ such that any $\mu(a)$ adjacent ones form a distinguished basis. Moreover, the monodromy operator applied to $v_i$ is $v_{i-\mu(a)}$, and the Seifert matrix of any one of these distinguished bases is given by the matrix in Equation \ref{seifertmatrix}. 
\end{theorem}

We are able to construct this basis very concretely as matching cycles using a Lefschetz fibration on the Milnor fiber (see the ``petals'' below in Section \ref{ss1.2}). The $\mu(a)$th root of the monodromy operator that made an appearance in various works (see Remark 3.5 of \cite{AT}) arises from a diffeomorphism of the Milnor fiber in our framework (this diffeomorphism was already considered in Section 2 of \cite{ebelinggusein}). We can also derive the linear relation between $\mu(a)+1$ adjacent members of the collection in the theorem, which involves entries of the inverse of the Seifert matrix \ref{seifertmatrix}. 

Because $p_a:\mathbb{C}^n\to\mathbb{C}$ is tame in the sense of Broughton \cite{bro} and has no critical point other than the origin, we can directly work with \begin{align}M_a:=p_a^{-1}(1),\end{align} instead of the Milnor fiber. Let us define $\alpha_1',\alpha_2',\ldots,\alpha_{\mu(a)}'\in\mathbb{Z}$ via the polynomial equation \begin{align}\prod_{i=0}^{n}(1-t^{r_i})^{(-1)^{i+n}}= 1+ \alpha_1't+ \alpha_2't^2 +\ldots + a_{\mu(a)}'t^{\mu(a)}.\end{align}  Here is a sharper version of Theorem \ref{s1thm} in light of these comments.

\begin{theorem}\label{s1thm2}
There exists a diffeomorphism $\phi:M_a\to M_a$ and a homology class $v\in H_{n-1}(M_a,\mathbb{Z})$ with the following properties.
\begin{enumerate}
    \item $\phi^{r_n}$ is the identity.
    \item $\phi^{-\mu(a)}$ acts as the monodromy operator on $H_{n-1}(M_a,\mathbb{Z})$.
    \item $\phi_*:H_{n-1}(M_a,\mathbb{Z})\to H_{n-1}(M_a,\mathbb{Z})$ preserves the Seifert form.
    \item For any integer $k$, the ordered collection $\phi^{k}_*v,\ldots
    ,\phi^{k+\mu(a)-1}_*v$ forms a distinguished basis of $H_{n-1}(M_a,\mathbb{Z})$.
    
    \item The Seifert matrix of $v,\ldots ,\phi^{\mu(a)-1}_*v$ is the matrix in Equation \ref{seifertmatrix}.
    \item We have the relation $$v+\alpha_1'\phi_*v\ldots +\alpha_{\mu(a)}'\phi^{\mu_n}_*v=0$$ in $H_{n-1}(M_a,\mathbb{Z})$.
\end{enumerate}
\end{theorem}

\begin{remark}
As a corollary of (2), (5) and (6), we obtain a geometric proof of the matrix identity (see \cite{horo},\cite{hertling},\cite{AT})
$$M(a)^{\mu(a)}=(-1)^nS(a)^{-1}S(a)^T, $$ where $M(a)$ is the companion matrix:
\begin{equation}
\begin{bmatrix} 
-\alpha_1' & 1 & 0 & 0 & \ldots & 0 \\
-\alpha_2' & 0  & 1 & 0&\ldots & 0\\
-\alpha_3' & 0 & 0 &1 &\ldots&0\\
 &&&\ldots&&&\\
-\alpha_{\mu(a)-1}' &0&0&\ldots&0&1\\
-\alpha_{\mu{(a)}}' &0&0&0&\ldots&0\\
\end{bmatrix}.
\end{equation}
\end{remark}

Before we explain this theorem and its proof further, we make a digression into the categorical version of Orlik-Randell conjecture. Our main goal is to explain the ``mirror computation'' of \cite{AT}.


For $a= (a_1,\ldots,a_n)\in \mathbb{Z}^n_{>0}$, let $a^T := (a_n,\ldots,a_1)$. Then the homological version of the Berglund-H\"{u}bsch mirror symmetry (see \cite{berg} and \cite{takahashi}) conjectures the following (Conjecture 1.1 from \cite{AT}). 

\begin{conjecture}\label{s1con} There exists a triangulated equivalence\[D^b Fuk (p_a) \simeq HMS^{L_{p_{a^T}}}_S(p_{a^T})\]
\end{conjecture}
The category on the right hand side (i.e. B-side) is a derived version of the maximally graded matrix factorization category. For details see Section 2 of \cite{AT}. The Grothendieck group of this category along with its Euler pairing was computed in Theorem 3.6 of the same paper.

Aramaki-Takahashi's computation starts by constructing an object $A$ and an automorphism $\Psi$ of $HMS^{L_{p_{a^T}}}_S(p_{a^T})$ such that $A, \Psi(A),\ldots,\Psi^{\mu(a)-1}(A)$ is a full exceptional collection. They then compute the matrix of the Euler pairing with respect to the basis given by the classes of these objects in the Grothendieck group. 

The category on the left side (i.e. the A-side) in Conjecture \ref{s1con} is the derived Fukaya-Seidel category. There are many ways to define this category, but we believe that the conjecture should be stated with one of the constructions, for example Example 2.20 from \cite{GPS} or possibly \cite{Fan}, that do not involve making the auxiliary choice of a Morsification, in line with the B-side category.

\begin{remark}It is well known that the Grothendieck group of $D^b Fuk (p_a)$ with its Euler pairing is isomorphic to $H_{n-1}(M_a,\mathbb{Z})$ with the Seifert form. In light of this, let us ``categorify'' Theorem \ref{s1thm2} (1)-(4).  The only thing missing to make the following statements rigourous is a discussion of gradings but this is outside of our scope.

We can construct an object $A$ and an automorphism $\Phi$ of $D^b Fuk (p_a)$ such that the collection $\Phi^{k}(A)$ (with some shifts that we are leaving unspecified here), $k\in\mathbb{Z}$, forms a helix of period $\mu(a)$ in the sense of \cite{helix}. Equivalently, $\Phi^{-1}$ is polarization of index $\mu(a)$ and $A,\ldots ,\Phi^{\mu(a)-1}(A)$ is a Lefschetz decomposition in the sense of \cite{kuznetsov}. Noting the relationship between the Serre functor of $D^b Fuk (p_a)$ and the action of monodromy, this covers parts (2)-(4). Part (1) would say that $\Phi^{r_n}$ acts as a shift.
\end{remark}

Our resolution of the Orlik-Randell conjecture along with Aramaki-Takahashi's computation shows that Conjecture \ref{s1con} holds at the Grothendieck group level. In fact, the full exceptional collections that are used in the two computations give rise to exactly the same Euler matrices, making it plausible to conjecture that these objects correspond to each other under a mirror equivalence\footnote{In order to observe the transposition taking place note the definition of $d_i$ from \cite{AT} and our definition of $r_i$.}. 

\begin{remark}
Motivated by the present work, we can propose an equivariant upgrade to the Conjecture \ref{s1con}. In particular, we expect the automorphisms that we mentioned on both sides of Conjecture \ref{s1con} to be mirror to each other. This work will appear in a separate paper for the more general case of invertible singularites.
\end{remark}

%
%

We will now explain in some detail the main actors in  Theorem \ref{s1thm2} and also the mechanism that leads to the computation of the Seifert matrix in part (5), which is our main contribution.

\subsection{A closely related Lefschetz fibration}\label{ss1.1}

We will use a special Morsification of $p_a$ given simply by $p_a + z_1$. It is easy to check that $p_a +\epsilon z_1: \mathbb{C}^n \to \mathbb{C}$ is a Lefschetz fibration with $\mu(a)$ critical points for every $\epsilon\in\mathbb{C}^*$. Using Proposition 2.5 of \cite{Fan}, we can show that $\mathbb{C}^{n+1}\to \mathbb{C}^2$ given by $(\epsilon,z_1,\ldots ,z_n)\mapsto (\epsilon,p_a +\epsilon z_1)$ is well-behaved at infinity.

\begin{remark}
We will define our notion of ``tame'' for holomorphic maps in Section \ref{sstame}. Proposition \ref{lemmatamefull} proves that all maps that appear in this introduction are tame (in our sense). Its proof relies on the existence of a fiberwise compactification by quasi-smooth divisors (see Appendix B of \cite{dimcabook} for an introduction to weighted projective spaces, and in particular for the definition of a quasi-smooth complete intersection). This approach is based on the standard technique of constructing fiberwise compactifications by smooth divisors using a Nash blow-up of a Lefschetz pencil, see Section 19b) of \cite{sbook} for example. 
\end{remark}

Our computation also shows that the critical values of $p_a +z_1: \mathbb{C}^n \to \mathbb{C}$ are equidistributed on a circle. The explanation for this nice behaviour is that a certain order $\mu(a)$ cyclic subgroup of $\Gamma_a$ makes $p_a +z_1: \mathbb{C}^n \to \mathbb{C}$ equivariant, where we use multiplication by $e^{\frac{2\pi i }{\mu(a)}}$ on the base (see Section \ref{ss3.2}).

\begin{remark}
It is easy to prove that $\Gamma_a$ is a graph over $\{\lambda_1^{a_1\ldots a_n}=\lambda^{\mu(a_2,\ldots , a_n)}\}\subset (\mathbb{C}^*)^2$, and hence is isomorphic to it.
\end{remark}

Moreover, after base changing by a finite cyclic branched cover $\mathbb{C} \to \mathbb{C}$ (Section \ref{ss2.5}), $p_a+z_1$ compactifies to a Lefschetz fibration over $\mathbb{P}^1$. Let us first introduce the ``other side'' of this fibration, which is where we actually do our computations.

For an $n$-tuple $a= (a_1,\dots ,a_n)\in\mathbb{Z}^n_{> 0}$, we define \begin{align}f_a :M_a=\{p_a(z_1,\ldots,z_n)=1\} \to \mathbb{C},\end{align}
where the map simply projects to the $z_1$-coordinate. $f_a$ is a Lefschetz fibration with $a_1\ldots a_n$ critical points whose critical values are again equidistributed on a circle. The explanation is that the $\{\lambda=1\}$ subgroup of $\Gamma_a$ is a cyclic group of order $a_1\ldots a_n$ which acts on $M_a$.  A particular generator of this action, which rotates the base $\frac{2\pi}{a_1\ldots a_n}$ degrees counter-clockwise, is the diffeomorphism $\phi:M_a\to M_a$ in Theorem \ref{s1thm2}.

\begin{remark}\label{relativeorlik}
The relative cohomology group that Orlik-Randell use in Theorem 2.11 of \cite{OR} is precisely the same as \[H^{\text{mid}}(\text{Tot}(f_{(a_1,\ldots,a_n)})=M_a , f^{-1}_{(a_1,\ldots,a_n)}(0)=M_{(a_2,\ldots ,a_n)}).\]
\end{remark}


Consider the weighted projective space $$\mathbb{P}(1,q_1,\ldots,q_n,1),$$ where $q_k$ are the weights of the $\mathbb{C}^*$-action on $\mathbb{C}^n$ described in Section \ref{ss3.1}. Let $s,x_1,\ldots ,x_n,t$ be the weighted homogenous coordinates. We define $$X_a:=\{s^{\mu(a)}x_1+p_a(x_1,\ldots ,x_n)=t^{a_1\ldots a_n}\}\subset\mathbb{P}(1,q_1,\ldots,q_n,1)-\{s=t=0\},$$ and the map \begin{align}\label{eqpi}\pi_a: X_a\to \mathbb{P}^1,\end{align} which projects to $[s:t]$, is the compactification we mentioned above. In the complement of $\{s=0\}$, and $\{t=0\}$ resp. we obtain the $a_1\ldots a_n$-fold branched cover of $p_a+z_1$, and the $\mu(a)$-fold cover of $f_{(1,a)},$ respectively. In Section \ref{ss3.2}, we give a more hands on description of $\pi_a$, which might be beneficial for the reader who is unfamiliar with weighted projective spaces.

\begin{remark}Note that the projection $\mathbb{P}(1,q_1,\ldots,q_n,1)-\{s=t=0\}\to \mathbb{P}^1$ is isomorphic to the fibration underlying the vector bundle $$\mathcal{O}(q_1)\oplus\ldots\oplus\mathcal{O}(q_n)\to \mathbb{P}^1.$$ 
\end{remark}

Using this relationship we show that the analysis of a certain basis of vanishing paths in the base of $p_a+z_1$ can be translated to $f_{(1,a)}$ (Section \ref{ss3.3}, in particular Figure \ref{basis}). 

\subsection{The analysis of vanishing paths for $f_{\tilde{a}}$}\label{ss1.2}
As a result of the previous section, we are interested in certain vanishing paths in the base of $f_{(1,a)}$. We present our results for any $\tilde{a}=(a_0,\ldots ,a_n)$, $a_0\geq 1$. Note that $f_{\tilde{a}}$ is obtained from $f_{(1,a_1,\ldots,a_n)}$ by pullback over the $a_0$-fold branched cover of $\mathbb{C}$.

First of all, note that the total space of $f_{(a_1,\ldots,a_n)}$, i.e. $M_{(a_1,\ldots ,a_n)}$ is equal to $f_{(a_0,\ldots,a_n)}^{-1}(0)$. Recall that the critical values of $f_{\tilde{a}}$ are $a_0\ldots a_n$ many points equidistributed on a circle. The vanishing paths in the base of $f_{\tilde{a}}$ that we are interested in are the radial ones that go from the critical values to the origin. We would like to present the isotopy classes of their vanishing spheres at the origin as matching spheres of $f_{(a_1,\ldots,a_n)}$.

\begin{definition}\label{defpetal}
Let $S$ be finite set of equidistributed points on a circle centered at $0$ in $\mathbb{C}$. Let $k<\abs{S}$ be a positive integer. We define a \textbf{petal} of length $k$ (or $k$-petal for short) to be an isotopy class of embedded paths \[([0,1], \{0,1\}) \to (\mathbb{C}, S)\] with the image of $(0,1)$ disjoint from $S$,
which admits a representative of the following form. First, take a path on the circle containing $S$ that starts at a point of $S$, goes counter-clockwise skipping over exactly $k$ arcs, and ends at another point of $S$; and then, slightly push it out of the circle. See the left side of Figure \ref{fpetals} for a depiction.\end{definition}

\begin{figure}
\includegraphics[width=0.8\textwidth]{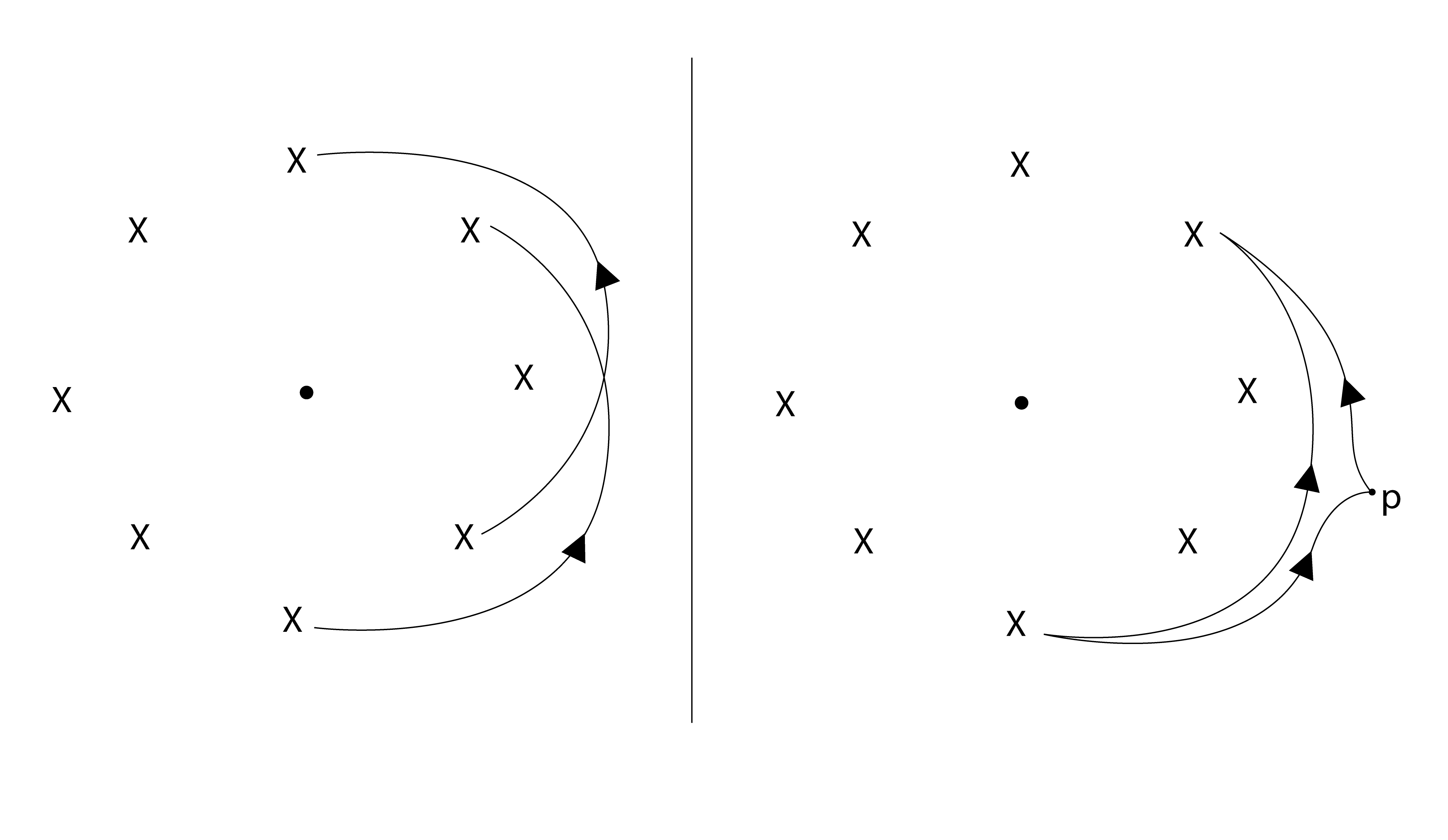}
\caption{On the left are two adjacent petals, and on the right are the two vanishing paths used to analyze a petal}
\label{fpetals}
\end{figure}
The following proposition is the main computation of the paper. It uses the Lefschetz bifibration technique explained in \cite{sbook}. The full discussion is in Section \ref{ss3.4}.

\begin{proposition}\label{s1prop1}
 The vanishing sphere of a radial path in the base of $f_{(a_0,\ldots,a_n)} $
 is isotopic to the matching sphere of a petal of length $\mu(a_2,\ldots,a_n)$ (equal to $1$ for $n=1$) in the base of $f_{(a_1,\ldots,a_n)}$. Moreover if the radial path is rotated $\frac{2\pi}{a_0\ldots a_n}$ degrees the corresponding petal rotates $-\frac{2\pi}{a_1\ldots a_n}$.
 \end{proposition}

\begin{remark}
Let us remark that using the Lefschetz fibration $X_a \to \mathbb{P}^1$, one can directly see that $\mu(a_2,\ldots,a_n)$-petals are matching paths in the base of $f_{(a_1,\ldots,a_n)}$.

In general, there might be multiple non-isotopic matching spheres over a matching path, even if there is only one critical point above each critical value (see Remark 16.11 of \cite{sbook}). Therefore, in this lemma we are abusing notation, but this is harmless. Our proof (and in general the Lefschetz bifibration technique once the bifibration is fixed) produces a particular matching sphere. Also see Remark \ref{matchinghomology}.
\end{remark}
 
 \begin{remark}
We discovered this result by experimenting on Mathematica. The short code we used is provided in Appendix \ref{appd}. We prove the proposition using results of Egervary from \cite{eger}. Unfortunately, we cannot access this paper, and also it is in Hungarian, which we cannot read. In \cite{trinomials} and \cite{szabo}, the results of \cite{eger} were summarized, and we will use them as such.
 \end{remark}

 \begin{remark}
As it is well-known in the numerical analysis community, the change of the roots of a polynomial can be ultra-sensitive to changes in the coefficients. For a popular discussion of this see \cite{wilkinson} or consult the Wikipedia page for ``Wilkinson polynomial''. We believe therefore that the question of how reliable computer generated ``movies'' of critical values actually are is not just a question of mathematical pedanticism.   
 \end{remark}
 
 We can take as our $v$ in Theorem \ref{s1thm2} any of the matching cycles of the $\mu(a_2,\ldots , a_n)$-petals in the base of $f_a:M_a\to\mathbb{C}$. Notice how $\phi:M_a\to M_a$ rotates these petals in the counter clockwise direction.
 
 \subsection{Dual systems of vanishing paths in $f_a$}\label{ss1.3}
 The final step in our analysis is to understand the intersection numbers of the matching spheres of two petals of length $\mu(a_2,\ldots a_n)$ in the base of $f_{(a_1,\ldots,a_n)}$ in terms of the intersection numbers of the vanishing cycles of radial paths in the same base. We note that these vanishing cycles live in $\{f_{(a_1,\ldots,a_n)}=0\}$, which is the total space of $f_{(a_2,\ldots,a_n)}$. Clearly, such an understanding leads to an inductive argument, which would eventually compute the Seifert matrix of $p_a+z_1$ with respect to the basis of vanishing paths we had chosen before.
 
 The key points here are the following. Let us fix a disk $D$ in the base of $f_{(a_1,\ldots,a_n)}$ which contains the circle of critical values in its interior. Let us also fix a point $p\in \partial D$. We have a variation pairing on \[H_{\text{mid}}(\text{Tot}(f_{(a_1,\ldots,a_n)} , f^{-1}_{(a_1,\ldots,a_n)}(p)).\]
 We refer the reader to \cite{svar}, Section 1.1 (also see our Appendix C) for how this bilinear pairing is constructed\footnote{note that there is a slight issue in the explanation of the right hand side of the Equation (1.4) there, but this is easily fixed.}.
 
 In the image of the injection \[H_{\text{mid}}(\text{Tot}(f_{(a_1,\ldots,a_n)} )  \xhookrightarrow{} H_{\text{mid}}(\text{Tot}(f_{(a_1,\ldots,a_n)} , f^{-1}_{(a_1,\ldots,a_n)}(p)),\]
 the variation pairing agrees with the usual intersection pairing on $H_{\text{mid}}(\text{Tot}(f_{(a_1,..,a_n)})$.
 
 Notice that a petal in the base of $f_{(a_2,\ldots,a_n)}$ can be represented as a linear combination of two Lefschetz thimbles (in homology). These are the thimbles of vanishing paths outside the circle of critical values. See the right side of Figure \ref{fpetals} for a depiction and Section \ref{ss2.7} for more details.

 The intersection numbers of the vanishing cycles of the radial vanishing paths may equally well be considered inside $f^{-1}_{(a_2,\ldots,a_n)}(p)$ by moving the origin to $p$ and dragging the paths with it, as in Figure \ref{fmovingorigin}.
 
\begin{figure}
\includegraphics[width=0.8\textwidth]{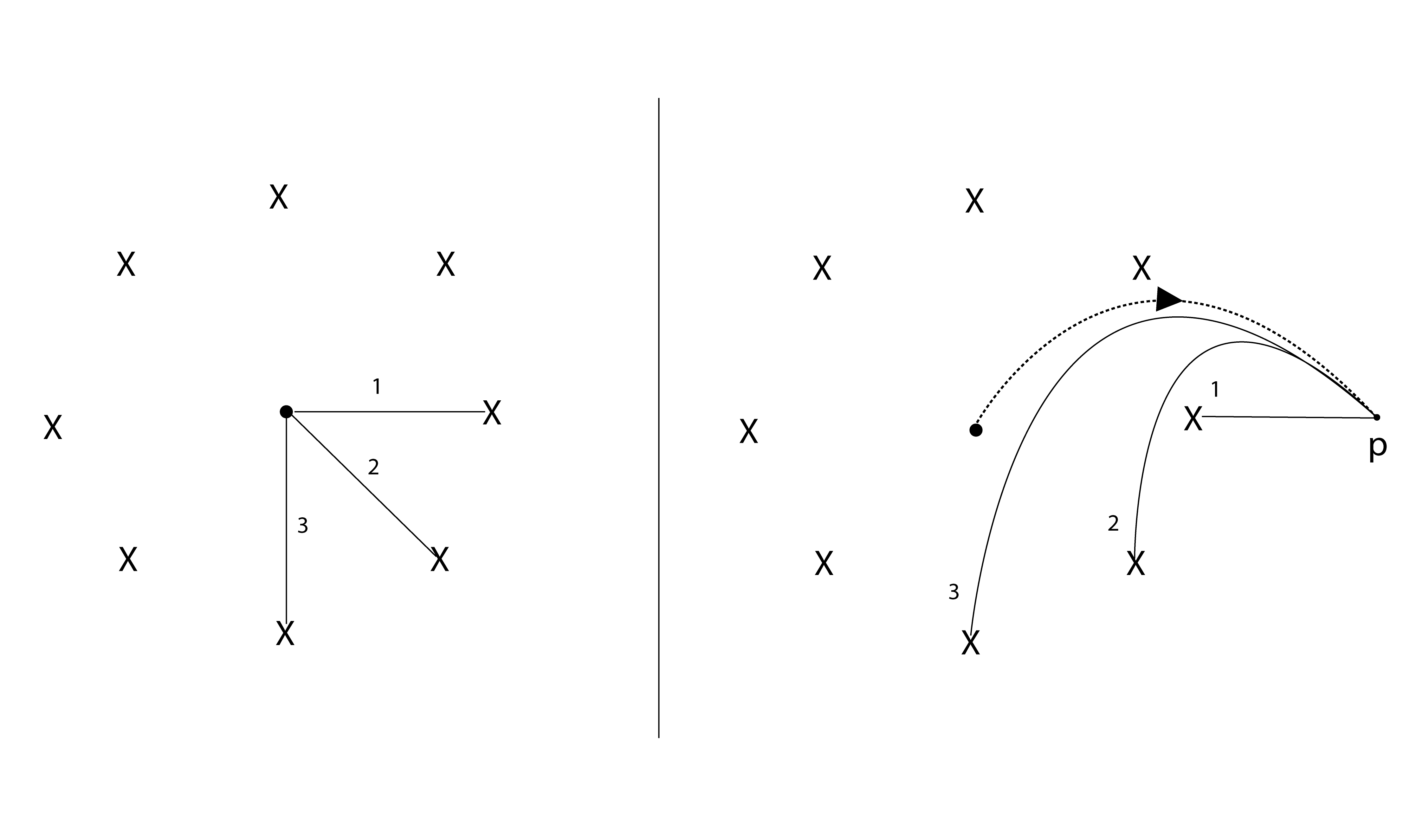}
\caption{Dragging the radial vanishing paths outside the circle.}
\label{fmovingorigin}
\end{figure}
 
 The final step involves relating the intersection numbers of the two systems of vanishing paths as in Figure \ref{fkoszul}.
 
\begin{figure}
\includegraphics[width=0.8\textwidth]{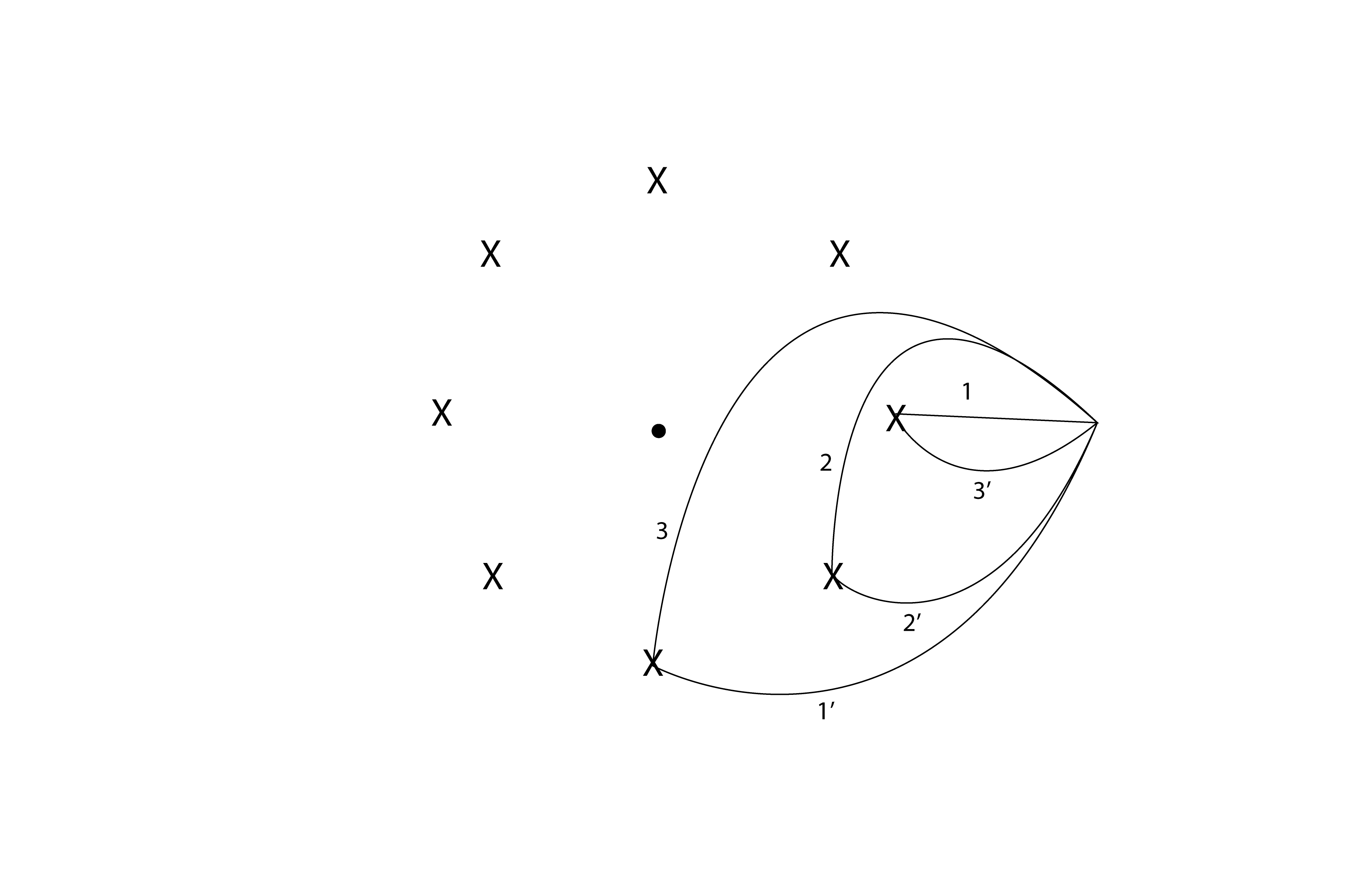}
\caption{Dual systems of vanishing paths.}
\label{fkoszul}
\end{figure}
 
 These two systems are ``dual'' to each other as in Figure 18.4 of \cite{sbook}, which is the basis of our induction (Section \ref{ss2.6}).
 
 \subsection{An inductive description of the matrices in the Orlik Randell conjecture}\label{ss1.4} Let $N_k$ be the $k\times k$ regular nilpotent matrix\[(N_k)_{i,j} = \delta_{i,j-1}.\] 
 
 \begin{definition}\label{s1defrainbow} A $k\times k$ matrix $A$ is called a \textbf{rainbow matrix} if there exists integers $\beta_1,\ldots,\beta_{k-1}$ such that: 
 \[A = Id_{k} + \sum^{k-1}_{i=1} \beta_i\cdot N^i_k.\]
We call $\beta_1,\ldots,\beta_{k-1}$ the \textbf{colors} of $A$. 
\end{definition}
Note that the inverse of a rainbow matrix is also a rainbow matrix. 
Hence, the matrix in Equation \ref{seifertmatrix} is equivalently defined as \begin{align}S(a_1,\ldots,a_n):= \prod_{i=0}^n (1-N_{\mu(a)}^{r_i})^{(-1)^{i+n-1}},\end{align}
using the substitution homomorphism $\mathbb{C}[[t]]\to Mat(\mu(a),\mu(a))$, sending $t$ to $N_{\mu}$.

Let $k<l$ be positive integers. We define the \textbf{$l$-rainbow extension }of a $k \times k$ rainbow matrix with colors $\beta_1,\ldots,\beta_{k-1}$ to be the $l\times l $ rainbow matrix with colors $\beta_1,\ldots,\beta_{k-1},0,\ldots,0$.

We have the following inductive description of matrices $S(a_1,\ldots,a_n)$, for $a_i\geq 2$,
\begin{itemize}
    \item $S(a_1) = Id - N_{a_1-1}.$ 
    \item For $n>1$, $S(a_1,\ldots,a_n) $ is obtained from $S(a_2,\ldots,a_n)$ in three steps.
    \begin{enumerate}
        \item invert $S(a_2,\ldots,a_n)$ to get $A$, which is a rainbow matrix.
        \item $\mu(a_1,\ldots,a_n)$-rainbow extend $A$ to get $B$.
        \item change the $\mu(a_2,\ldots,a_n)^{\text{th}}$color of $B$ from 0 to $(-1)^n$ to get $S(a_1,\ldots,a_n)$.
    \end{enumerate}
    \end{itemize}
    \begin{example} To get $S(2,3)$ from $S(3)$ we use the following procedure:
    
    \begin{itemize}
        \item $S(3) = \begin{pmatrix}
        1&-1\\
        0&1\\
        \end{pmatrix}$
        \item $A = \begin{pmatrix}
        1& 1\\
        0& 1\\
        \end{pmatrix}$
        \item$B = \begin{pmatrix}
        1&1&0&0\\
        0&1&1&0\\
        0&0&1&1\\
        0&0&0&1\\
        \end{pmatrix}$
        \item$S(2,3) = \begin{pmatrix}
        1&1&1&0\\
        0&1&1&1\\
        0&0&1&1\\
        0&0&0&1\\
        \end{pmatrix}$
    \end{itemize}
    \end{example}
   
    We will of course explain the details of how the two inductive procedures that we explained in Sections \ref{ss1.3} and \ref{ss1.4} correspond to each other. 
    
    \begin{remark}
    We are currently investigating how to categorify this inductive description. From a conceptual viewpoint the trickiest part is the step (3) from the inductive procedure. The results of \cite{ssus} are relevant here.
    \end{remark}
    
%
%
%
%
    Here is a summary of the paper. In Section \ref{s2}, we discuss some definitions and basic statements from Picard-Lefschetz theory. In Section \ref{s3}, we analyze the fibrations $p_a+z_1$ and $f_a$, including their relationship to each other. We prove Proposition \ref{s1prop1} and set the stage for the computation of Seifert matrices. In Section \ref{s4}, we finish the proof of Theorem \ref{s1thm2}.
    
    There are two appendices. Appendix A is about explicit computations of critical values and points. In Appendix B, we provide our Mathematica code used in discovering Proposition \ref{s1prop1}. Upon the referee's request, we have also added an Appendix C about the construction of the variation pairing in the first revision.
    
\subsection{Acknowledgements} I thank Yanki Lekili for discussions regarding Remark \ref{longremark}, and the anonymous referee for their careful reading of the paper.

\section{Picard-Lefschetz theory}\label{s2}

For this section, we assume a familiarity with Picard-Lefschetz theory as discussed in Sections 15d,e and 16 of \cite{sbook}. Even though we do not really take into account the symplectic geometry and do not touch upon Floer theory, which makes this reference a bit of an overkill, we are not aware of a more elementary reference that discusses matching cycles and Lefschetz bifibrations. The reader would also need to be familiar with more classical notions from singularity theory as explained for example in Chapter 5 of \cite{ebeling}.  As mentioned we will also use the variation pairing discussed in Section 1.1 of \cite{svar}. We will redefine the notions we use in this paper for clarity, but the proofs will be omitted.

\subsection{Sign conventions on intersection form in middle dimension}\label{ss2.1}
Let $X$ be a smooth complex $n$-dimensional affine variety with its natural orientation. Let us denote the intersection pairing on its middle dimensional homology $H_n(X)$ by $<v,w>$. We  define: $$v\cdot w:=(-1)^{\frac{n(n+1)}{2}}<v,w>\in\mathbb{Z}, \text{ for } v,w\in H_n(X).$$

If $v$ and $w$ are represented by closed oriented submanifolds $V$ and $W$, which are transverse to each other, then the sign of an intersection is the sign of the ordered basis $$(w_1,v_1,\ldots,w_n,v_1)$$
for the $v\cdot w$ pairing, where $(v_1,\ldots ,v_n)$ and $(w_1,\ldots ,w_n)$ are positively oriented bases of $V$ and $W$ at the intersection point.

For an explanation of this sign change see Convention 1.3 of \cite{svar}. 


\subsection{Tame maps}\label{sstame} Let us consider a map of smooth manifolds $g: X\to B$, which is a submersion away from a compact set $S\subset X$. Let us take an Ehresmann connection (i.e. a choice of a horizontal subbundle) $E$ on $X-S\to B$. We call $(g,S, E)$ \textbf{tame}, if for any complete vector field on $B$, its unique lift to $X-S$ satisfies the property that every integral curve $\gamma: [0,\epsilon)\to X-S$, $\epsilon\in\mathbb{R}_{>0}$, extends to a continuous map $[0,\epsilon]\to X$. In this paper, we will always assume that $S$ is the set of critical points of $g$.

Of particular importance to us are the Ehresmann connections defined by connection $2$-forms. Namely, assume that we have a $2$-form $\Omega$ on $X$, whose restriction to $g^{-1}(b)\cap (X-S)$ is non-degenerate for every $b\in B$. Taking $\Omega$-kernels to tangent spaces of fibers we obtain an Ehresmann connection on $X-S\to B$. 

A slight modification of this constrution is as follows. Let us assume that we have a one form $\Theta$ on $X$ whose restriction to $g^{-1}(b)\cap (X-S)$ is a contact form for every $b\in B$. Then we can consider the kernel of $d\Theta$ restricted to the distribution defined by $\Theta$ to obtain an Ehresmann connection.

\subsection{Vanishing and matching paths}\label{ss2.2}

Let $f: X^{2n}\to \Sigma$ be a Lefschetz fibration with finitely many critical points with different critical values. We will always be working in the cases where $X$ is a smooth complex quasi-projective variety, $\Sigma=\mathbb{C}$ or $\mathbb{P}^1$, and $f$ is a regular function. Moreover, $X$ will always come explicitly embedded in some weighted projective space cross affine space (disjoint from the orbifold points), and we will always use the restriction of the standard symplectic structure to $X$. We then use this symplectic structure to define an Ehresmann connection as in the previous section. Whenever we say Lefschetz fibration (or LF for short) we will assume all of these plus tameness as defined in the previous section.

Let us fix a regular value $p$ of $f$. Let $\gamma$ be an embedded path $[0,1]\to \mathbb{C}$ connecting $p$ to a critical value, which does not intersect the set of critical values in its interior. We call such a path a \textbf{vanishing path}. 

A vanishing path $\gamma$ determines a sphere $S_{\gamma}$ (called the \textbf{vanishing sphere}), which is a Lagrangian submanifold of $f^{-1}(p)$. If $\gamma$ is homotopic to $\gamma'$ among vanishing paths, then $S_{\gamma}$ and $S_{\gamma'}$ are Hamiltonian isotopic to each other. We call the homology class of $S_{\gamma}$, which is determined only up to a sign, the \textbf{vanishing cycle}. Once an orientation is chosen we call the resulting class, the \textbf{oriented vanishing cycle}.

Dependence of the isotopy class of $S_{\gamma}$ on the homotopy class of the vanishing path $\gamma$ is controlled by the Picard-Lefschetz transformations. The main point is that if $S$ is a Lagrangian sphere in $f^{-1}(p)$, and $\alpha_{\gamma}$ is a loop starting and ending at $p$, and tightly encircling $\gamma$ in the counterclockwise direction, then the parallel transport of $S$ along $\alpha_{\gamma}$ is Hamiltonian isotopic to the right handed (positive) Dehn twist of $S$ along $S_{\gamma}$: $\tau_{S_{\gamma}}S$.

If we orient $S$ in the above discussion then $\tau_{S_{\gamma}}S$ gets an induced orientation, and we have the Picard-Lefschetz formula:
\begin{align*}
[\tau_{S_{\gamma}}S] =[S]-([S_{\gamma}]\cdot [S])[S_{\gamma}].
\end{align*}

Another geometric object that we can assign to a vanishing path is its \textbf{Lefschetz thimble}, which is an embedded Lagrangian disk in $X$ with boundary on $f^{-1}(p)$. This can be thought of as the union of the vanishing spheres for all the points on the vanishing path. We call the homology class (up to a sign) it defines in $H_n(X,f^{-1}(p))$ the \textbf{Lefschetz cycle} and after an orientation (i.e. a generator of $H_n(D^n,\partial D^n)$) is chosen we call it  \textbf{oriented Lefschetz cycle}.

Finally, if we have an embedded path $m:[0,1]\to \mathbb{C}$ connecting a critical value to a different critical value, which does not intersect the set of critical values in its interior, so that the two vanishing spheres at $m(1/2)$ are isotopic to each other, then we call $m$ a \textbf{matching path}. Given the homotopy class of such an isotopy of spheres, we can talk about the \textbf{matching cycle}, which is an element of $H_n(X)$ up to the usual sign ambiguity; and an \textbf{oriented matching cycle}, which is an honest element of $H_n(X)$.

\begin{remark}\label{matchinghomology}
 In this paper, the fibers of all LF's will always be affine varieties. In this case, because of the Lefschetz theorem, the matching cycle is actually independent of the choice of the isotopy of spheres.
\end{remark}


\subsection{Distinguished bases}\label{ss2.3}

Let $f: X^{2n}\to \mathbb{C}$ be a Lefschetz fibration, $p$ be a regular value, and $l:[0,\infty)\to \mathbb{C}$ be a proper embedding, where $l(0)=p$ and $l$ does not intersect the set of critical values. We call such an $l$ a \textbf{branch cut} at $p$.

Now assume that there are $k$ critical values of $f$ and let $\gamma_1,\ldots \gamma_k$ be a sequence of vanishing paths (with respect to $p$), one for each critical value, which in addition do not intersect each other and $l$ in their interior and are pairwise transverse to each other at $p$. Moreover, the order in which the vanishing paths are written (left to right) matches the clockwise (starting at $l$) order of the directions along which the paths approach $p$. We call this a \textbf{basis of vanishing paths}.  It follows from the Lefschetz theorem that the corresponding Lefschetz cycles provides a basis of $H_n(X,f^{-1}(p))$ (well defined up to the usual $\pm$ indeterminacy of each cycle). We call this a \textbf{distinguished basis}. Of course, the distinguished basis only depends on the isotopy class of the basis of vanishing cycles.

If we have a sequence of vanishing paths (for $p$ and $l$) all with different critical values that is an ordered subsequence of a basis of vanishing paths, we call it a \textbf{system of vanishing paths}.

The primary operations to change the isotopy class of a basis of vanishing paths are the left and right mutations. These both modify adjacent (in the linear order) pairs of paths and they are inverses of each other. Note that we always make sense of left and right looking at the linear order on the page of $\gamma_1,\ldots \gamma_k$. If the path that changed moves to the left we call it the left mutation. In our conventions (same as \cite{sbook}), for the left mutation the changed vanishing sphere is modified by a positive Dehn twist, and for the right mutation by the inverse of a positive Dehn twist. 

Given a system of vanishing paths $\gamma_1,\ldots \gamma_r$, and orientations of Lefschetz thimbles $T_{\gamma_1},\ldots T_{\gamma_r}$, we can define the corresponding \textbf{Seifert matrix} by 

\[
    Seif_{i,j} \left( \gamma_{1}, \ldots , \gamma_{r} \right) = 
\begin{cases}
\lbrack S_{\gamma_i} \rbrack \cdot \lbrack S_{\gamma_j} \rbrack, &  i < j\\
1,  & i=j \\
0 , & \text{otherwise,}
\end{cases}
\] where we used the boundary orientation for the vanishing spheres $S_{\gamma_1},\ldots S_{\gamma_r}$.

If $\gamma_1,\ldots \gamma_k$ is basis of vanishing paths, we define, for every $i$ and $j$, $$[T_i]\cdot_l [T_j]:= Seif_{i,j} \left( \gamma_{1}, \ldots , \gamma_{r} \right),$$ and extend $\cdot_l$ to a bilinear pairing on $H_n(X,f^{-1}(p))$. We call $Seif(\gamma_{1}, \ldots , \gamma_{k})$ the \textbf{Seifert matrix} of $\gamma_1,\ldots \gamma_k$.

\begin{lemma}\label{lemmaseifert}
\begin{itemize}
\item The bilinear pairing $\cdot_l$ on $H_n(X,f^{-1}(p))$ does not depend on the choice of the basis of vanishing paths and the orientations of the Lefschetz thimbles. We call this the \textbf{Seifert pairing}.
\item On the image of the injection $H_n(X)\to H_n(X,\pi^{-1}(p))$ $\cdot_l$ agrees with the pairing $\cdot$. \qed
\end{itemize}

\end{lemma}

The proof of this lemma can be obtained using the Picard-Lefschetz formula, but there is a better way to see it using the variation pairing as defined in Section 1.1 of \cite{svar} and also our Appendix C. Note that the variation pairing and the Seifert pairing are the same pairing in the end, but the variation pairing can be defined without any reference to bases of vanishing paths.

\begin{remark}
Let us fix a regular value $p\in\mathbb{C}$ as above. To define the variation pairing one fixes a disk with $p$ on the boundary, which contains all critical values. Up to homotopy, the choices of such disks are in a natural one-to-one correspondance with the choices of branch cuts as above. It is clear how to associate a branch cut to a given disk using the Jordan curve theorem. In the other direction, given a branch cut $l$, we take a disk that contains all the critical points and $p$ whose boundary intersects $l$ at one point. Then, we remove from this disk a narrow tube which follows the branch cut starting from $p$. It is clear that (up to homotopy) these operations are inverses of each other. A slightly more careful argument shows that in fact there is a one-to-one correspondance between pairs of branch cuts and admissible systems of vanishing paths, and pairs of disks and systems of vanishing paths that lie inside the disk. Ultimately the disk is a more convenient choice, as it makes sense for LF's over more general bases.

We also note that the Seifert pairing, or equivalently, the variation pairing does not depend on the choice of $p$ and $l$ (or $p$ and the disk) either. Given another choice $p'$ and $l'$ we can easily find an isotopy of the base relative to the critical values that takes $p$ and $l$ to $p'$ and $l'$. We "drag" the basis of vanishing paths along to identify $H_n(X,f^{-1}(p))$ with $H_n(X,f^{-1}(p'))$ and compare the two Seifert pairings using the first part of Lemma \ref{lemmaseifert}. Note that the isotopy is not canonical, and hence the identifications are not canonical. The situation is similar for the variation pairing.
\end{remark}

In the case where $X=\mathbb{C}^n$, the homology exact sequence shows that $H_n(X,f^{-1}(p))$ is canonically isomorphic to $H_{n-1}(f^{-1}(p))$. This is used to construct the \textbf{Seifert form} on $H_{n-1}(M_a=p_a^{-1}(1))$ with the notation as in the introduction. Using the variation pairing approach, one can define the Seifert form directly by taking a disk in the base of $p_a:\mathbb{C}^n\to \mathbb{C}$ that contains the origin; there is no need to perturb (see Appendix C). To use the Seifert pairing approach, one needs to consider the one parameter ($\epsilon \in\mathbb{C}$) family of tame holomorphic maps $p_a^{\epsilon}:= p_a+\epsilon z_1 :\mathbb{C}^n\to \mathbb{C}$. For $\epsilon\neq 0$, we have our Seifert form on the regular fibers of $p_a^{\epsilon}$. Using parallel transport one can then define the desired form on $H_{n-1}(M_a)$. The independence on the path boils down to again part (1) of Lemma \ref{lemmaseifert}.

Finally, the notion of a distinguished basis in $H_{n-1}(M_a)$ is perhaps the most confusing part of the story as it cannot be defined without perturbing to a Lefschetz fibration. Since $p_a+z_1:\mathbb{C}^n\to \mathbb{C}$ is a Lefschetz fibration it is easy to define this notion on $H_{n-1}((p_a+z_1)^{-1}(A))$ for $A$ a regular value. Now again one has to use the parallel transport diffeomorphisms in the previous paragraph to get the desired notion on $H_{n-1}(M_a)$.

\subsection{Lefschetz fibrations with cyclic symmetries}\label{ss2.4}

Let $\psi_n:\mathbb{C}\to \mathbb{C}$ be the map $z\mapsto e^{2\pi i/n}z$, for $n\in\mathbb{Z}_{> 0}$. Also let $X\subset \mathbb{C}^n$ be an affine variety, where $\mathbb{C}^n$ has coordinates $x_1\ldots ,x_n$. Then a map $X\to X$ is called a \textbf{diagonal automorphism} if it is the restriction of a map $(x_1,\ldots x_n)\mapsto (\xi_1x_1,\ldots \xi_nx_n)$, where $\xi_i\in\mathbb{C}$ with $|\xi_i|=1$.

We call a LF $\pi: X\to \mathbb{C}$ and a diagonal automorphism $\Phi:X\to X$ an \textbf{LF with cyclic symmetry of order $n$}, if\begin{itemize}
\item $\Phi^n=1$
\item $\pi\circ \psi_n=\Phi\circ \pi$
\item $0\in\mathbb{C}$ is a regular value of $\pi$
\item $\pi$ has $n$ critical values.
\end{itemize}

Now, let $\pi: X\to \mathbb{C}$ and $\Phi:X\to X$ be an LF with cyclic symmetry of order $n$. Moreover, let $\rho: \pi^{-1}(0)\to \mathbb{C}$ be an LF such that $\Phi^{-1}|_{\pi^{-1}(0)}$ is a cyclic symmetry of order $m$. Note that $m$ necessarily divides $n$.

Let us choose a radial branch cut $l$ from the origin. Let $\gamma_1,\ldots \gamma_n$ be the basis of vanishing paths, which are all radial. Assume that the vanishing cycle $\Delta_1$ of $\gamma_1$ is the matching cycle of a matching path $m_1$ in the base of $\rho$. The following is straightforward.

\begin{lemma} The vanishing cycle $\Delta_i$ of $\gamma_i$, for $1\leq i\leq n$, is the matching cycle of the matching path $m_i:=\psi^{-i+1}_n(m_1)$ in the base of $\rho$.\qed\end{lemma}

When we talk about the Seifert matrix of $\gamma_1,\ldots \gamma_n$, we choose the orientations such that the Lefschetz thimble $T_1$ of $\gamma_1$ is oriented arbitrarily, and then $T_i$ of $\gamma_i$ is oriented  using the fact that $T_i=\Phi^{i-1}(T_1)$. It is easy to see that the resulting Seifert matrix $Seif:=Seif(\gamma_1,\ldots \gamma_n)$ is independent of the orientation of $T_1$. Let us call this our \textbf{orientation convention}. The following is again straightforward.

\begin{lemma}
Let $N$ be the regular nilpotent matrix of size $n$, and also let for $2\leq i\leq n$, $$\alpha_i=\Delta_1\cdot \Delta_i.$$ Then, we have\[\pushQED{\qed} 
Seif=Id+\alpha_2N+\ldots \alpha_n N^{n-1}.\qedhere
\popQED\]
\end{lemma}

\subsection{Pullbacks of Lefschetz fibrations by branched covers}\label{ss2.5}

Let $X$ be a smooth complex affine variety and $\pi : X \rightarrow \mathbb{C}$ be a Lefschetz fibration. Let us also assume that the critical values of $\pi$ are $k^{th}$ roots of unity for some positive integer $k$, for simplicity.

Let $pow_{d} : \mathbb{C} \rightarrow \mathbb{C}$ be the map $z \mapsto z^{d}$, for $d$ a positive integer. Then, there exists a unique affine variety $X_{d}$ and a Lefschetz fibration $\pi_{d} : X_{d} \rightarrow \mathbb{C}$ with a map $X_{d} \rightarrow X$ such that:

\begin{align}
\xymatrix{ 
X_{d}\ar[r]\ar[d]_{\pi_d}& X \ar[d]\\ \mathbb{C}\ar[r]_{pow_d}&\mathbb{C},}
\end{align}is a pullback diagram of affine varieties.
%
%

\begin{example} Let $X = \{ x f_{1} \left( y_{1}, \ldots,y_{n}\right) + f_{2} \left( y_{1}, \ldots , y_{n} \right) = 0 \} \subset \mathbb{C}^{n+1}_{x,\vec{y}}$ and $\pi$ is the projection to the $x$ coordinate. Then $X_{d} = \{ z^{d} f_{1} \left( y_{1}, \ldots,y_{n}\right) + f_{2} \left( y_{1}, \ldots , y_{n} \right) = 0 \} \subset \mathbb{C}^{n+1}_{z,\vec{y}}$ and $\pi_{d}$ is the projection to the $z$ coordinate. \end{example}

We now list some straightforward properties of $\pi_{d} : X_{d} \rightarrow \mathbb{C}$.
\begin{itemize}
\item The critical values of $\pi_{d}$ are the $\left( d \cdot k \right)^{th}$ roots of unity
\item $\pi_{d}^{-1} \left(0\right) = \pi^{-1} \left(0\right)$
\item The vanishing cycle of a radial path from a critical value to the origin in the base of $\pi_{d}$ is exactly the same as the arc of the radial path that is its image under $pow_{d}$.
\item There is an action of $\mathbb{Z}/d\mathbb{Z}$ on $X_d$ which makes $\pi_d$ equivariant.
\end{itemize}




%
%

\subsection{Duality}\label{ss2.6}

Let $\pi : X \rightarrow \mathbb{C}$ be a Lefschetz fibration with $k$ critical points with values on a circle centered at the origin.

Let $p_{1},\ldots, p_{k}$ be the critical values where $p_{i+1}$ is the critical value clockwise adjacent to $p_{i}$, for $1 \le i \le k-1$. Let us also take a radial ray $l$ from the origin passing through the open arc between $p_{1}$ and $p_{k}$ as the branch cut.

Let $1\leq r \leq k$ and take radial paths $\gamma_{i}$ from $p_{i}$ to the origin, for each $1\le i \le r$,  as a system of vanishing paths. We also take a different system of vanishing paths $\widetilde{\gamma_{1}}, \ldots , \widetilde{\gamma_{r}}$ from the same critical values as follows:
\begin{itemize}
\item $\widetilde{\gamma_{i}}$ starts at $p_{r-i+1}$.
\item $\widetilde{\gamma_{i}}$ intersects the circle of critical values at exactly one point, which is required to be along the arc going clockwise from $l$ to $p_{1}.$
\end{itemize}

\begin{lemma}\label{s2lemma}
 Orient the Lefschetz thimbles of $\gamma_{1}, \ldots ,\gamma_{r}$ arbitrarily. Then, there exists a way of orienting the Lefschetz thimbles of $\widetilde{\gamma_{1}}, \ldots , \widetilde{ \gamma_{r}}$ such that the corresponding Seifert matrices $Seif$ and $\widetilde{Seif}$ are so that $\left( Seif^{-1}\right)_{i, j} = \left( \widetilde{Seif} \right)_{r-j+1, r-i+1}$.
 \end{lemma}
 
Again this can be seen using Picard-Lefschetz formula but a more conceptual proof is via the variation pairing. Let's move the base point outside the circle as described in Figure \ref{fkoszul} for clarity. The first thing that should be noted is that the two sets of thimbles span the same subspace of the relative homology. This follows because we can choose a smaller disk which contains all the critical values and the paths in question, contains no other critical point, and has the base point still at its boundary. The rest of the argument can be constructed by looking at the intersections (after the variation) of the two sets of Lefschetz thimbles from Figure \ref{fkoszul} and some linear algebra.


\subsection{Representing matching cycles as sums of Lefschetz cycles}\label{ss2.7}

Let $\pi: X\to \mathbb{C}$ be a $LF$. Let $p$ be a regular value with a branch cut $l$, and let $\gamma$ and $\gamma'$ be vanishing paths coming out of two different critical values $c$ and $c'$. Assume that the vanishing spheres of $\gamma$ and $\gamma'$ are isotopic. Then the smoothing $m$ of the concatanation of $\gamma$ and $\gamma'$-traced-backwards that does not intersect $l$ is a matching path from $c$ to $c'$.

We orient the Lefschetz thimbles $T$ and $T'$ of $\gamma$ and $\gamma'$ such that the boundary orientations of the vanishing spheres give rise to the same oriented vanishing cycles. Note that there are two ways of doing this. Then $[T]-[T']$ is a class in $H_n(X,\pi^{-1}(p))$ that is well defined up to sign. 

\begin{lemma}
\begin{itemize}
\item $[T]-[T']$ is in the image of the inclusion $H_n(X)\to H_n(X,\pi^{-1}(p))$.
\item $[T]-[T']$ as an element of $H_n(X)$ is the matching cycle of $m$.
\end{itemize}
\end{lemma}


\subsection{Lefschetz bifibrations to compute vanishing cycles as matching cycles}\label{ssbifib}

Let $\pi:X\to \mathbb{C}$ be an LF. In the context of this paper, a map $w: X\to \mathbb{C}^2$ is called a \textbf{compatible Lefschetz bifibration} if $\pi=pr_1\circ w$ and the following conditions are satisfied

\begin{enumerate}
    \item $w$ with the induced Ehresmann connection on $X-crit(w)$ is tame in the sense of Section \ref{sstame}.
    \item For each $c\in \mathbb{C}$ which is not a critical value of $\pi$, the map $$pr_2\circ w\mid_{\pi^{-1}(c)}: \pi^{-1}(c)\to\mathbb{C}$$ is a Lefschetz fibration.
    \item For each critical value $c$ of $\pi$, with its unique critical point $x$, the restriction of $Hess(\pi)_x$ to $ker(dw_x)$ is non-degenerate. Moreover,$$pr_2\circ w\mid_{\pi^{-1}(c)-x}: \pi^{-1}(c)-x\to\mathbb{C}$$ has only non-degenerate critical points.
    \item The critical point set $crit(w)$ of $w$ is smooth and $dw\neq 0$.
    \item $w\mid_{crit(w)}$ is a smooth embedding.
\end{enumerate}

Let us also assume that we do not have any fake critical points (see page 220 of \cite{sbook}). If we choose a path $\gamma$ from a regular point $p$ of $\pi$ in $\mathbb{C}$ to a critical point, i.e. a vanishing path, and consider the trajectories of critical values of Lefschetz fibrations from (2), we see a collision of two critical values at the end. Using Lemma 16.15 of \cite{sbook}, this lets us compute the vanishing cycle of $\gamma$ as a matching cycle of $pr_2\circ w\mid_{\pi^{-1}(p)}: \pi^{-1}(p)\to\mathbb{C}$. We refer to this procedure as \textbf{watching the movie} of critical values.

\subsection{Pencils in weighted-projective spaces}\label{sspencils}

Let $Y\in\mathbb{P}(1,1,w_1,\ldots ,w_n)$ be a quasi-smooth hypersurface with weighted homogenous coordinates $s,t,x_1,\ldots,x_n$ (Definition B17 of \cite{dimcabook}). 

Note that any weighted projective space admits a canonical symplectic structure (as a $V$-manifold) from the standard symplectic reduction construction \cite{weinstein} for the corresponding action of $U(1)$ with the given weights on complex affine space.

Consider the pencil $\mathcal{P}$ on $Y$ induced by intersecting the pencil $\bigcup_{[x:y]\in\mathbb{P}^1}\{xs=yt\}$ with $Y$. We furthermore assume that the base locus $B$ of $\mathcal{P}$ is quasi-smooth and away from $B$ the pencil has only isolated singularities. Note that the complement of $s=t=0$ in $\mathbb{P}(1,1,w_1,\ldots ,w_n)$ does not contain any orbifold points.

We can introduce the Nash blow-up of $Y$ along $B$, namely, $$\hat{Y}:=Y\times  \mathbb{P}^1\cap \{xs=yt\}\subset \mathbb{P}(1,1,w_1,\ldots ,w_n)\times  \mathbb{P}^1.$$

We have maps $\hat{Y}\to Y$ and $W: \hat{Y}\to \mathbb{P}^1$. The fibers of $W$ are the members of the original pencil $\mathcal{P}.$

Our goal is to show that $X:=Y-B\to \mathbb{P}^1$ is tame for the restriction symplectic form on $Y$. Note that this is a slight generalization of the standard set up in projective spaces, where we simply replaced smoothness requirements with quasi-smoothness.  

The key is to work not inside $\mathbb{P}(1,1,w_1,\ldots ,w_n)$, but at the canonical $U(1)$-bundle over it $$S^{2n+3}(1,1,w_1,\ldots ,w_n)\to \mathbb{P}(1,1,w_1,\ldots ,w_n),$$ and do everything $U(1)$-invariantly. For any subset $A$ of a weighted projective space, we denote its preimage by $A^{U(1)}$. Quasi-smoothness implies smoothness of preimages.

Note that $Y^{U(1)}$ has a natural $U(1)$-equivariant contact form $\theta$ obtained by restriction from the sphere. This induces a one form $\hat{\theta}$ on $\hat{Y}^{U(1)}$ which is contact on each fiber of $W^{U(1)}:= W\circ pr^{U(1)}$ (away from the critical circles) and is still $U(1)$-invariant. Using $\hat{\theta}$ we can construct a $U(1)$ invariant Ehresmann connection. Namely, we consider the $d\hat{\theta}$-kernel of the vertical part of $\hat{\theta}=0$ in the complement of the critical circles. This gives us the desired result because at any point $(p,s)$ in $B^{U(1)}\times \mathbb{P}^1\subset \hat{Y}^{U(1)}$ the horizontal subbundle is simply the tangent space to $\{p\}\times \mathbb{P}^1$.

Alternatively, we could pick any Ehresmann connection preserving $B^{U(1)}\times \mathbb{P}^1$, and use an averaging over $U(1)$ argument (see the proof of the Proposition 3 of \cite{dimca}).

\section{A Lefschetz bifibration picture of chain type singularities}\label{s3}

\subsection{Basics}\label{ss3.1}

In what follows (unless otherwise stated) we allow tuples to have length $0$. We denote the only tuple of length $0$ by $\emptyset$. Given a tuple ${{v}}$, we denote by $l_k{ v}$, $0\leq k\leq \abs{{v}}$, the tuple given by the last $k$ entries of ${{v}}$. Similarly, we define $f_k{ v}$, $0\leq k\leq \abs{{{v}}}$, to be the tuple given by the first $k$ entries of ${{v}}$.

Let $n\geq 1$, ${{a}}=(a_1,\ldots a_n)\in \mathbb{Z}_{>0}^{n}$. Recall that we defined \begin{align*} p_{a}(z_1,\ldots ,z_n):=z_1^{a_1}z_2+\ldots +z_{n-1}^{a_{n-1}}z_{n}+z_n^{a_n}.\end{align*} We consider the map $p_{{a}}:\mathbb{C}^n\to \mathbb{C}$ and its only singularity at the origin.

\begin{itemize}
\item The Milnor number of $p_{{a}}$ is $\mu({{a}})=a_1\ldots a_n-a_2\ldots a_n+\ldots +(-1)^{n-1}a_n+(-1)^n$. We define $\mu(\emptyset)=1$. Let us also define $d({{a}})=a_1\ldots a_n$ and $d(\emptyset)=1$. Note the recursive relation:
\begin{align}
\mu({{a}})+\mu(l_{\abs{{{a}}}-1}{{a}})=d({{a}}),
\end{align} for positive length $a$.
\item There is a $\mathbb{C}^*$ action on $\mathbb{C}^n$ given by \begin{align}\label{eqquasi}
z_k\mapsto t^{\mu(l_{n-k}{ a})d(f_{k-1}{a})}z_k
\end{align}
If we consider the action of $\mathbb{C}^*$ given by multiplication with $t^{d({{a}})}$ on the base, $p_{{a}}$ becomes equivariant.
\end{itemize}

\subsection{A $\mathbb{Z}/\mu({{a}})\mathbb{Z}$ equivariant Morsification}\label{ss3.2}
We consider a Morsification of $p_{{a}}$, given by $z_1+p_{{a}}:\mathbb{C}^n \to \mathbb{C}$. A straightforward computation (see Appendix \ref{appa}) shows that $z_1+p_{{a}}$ has $\mu({{a}})$ non-degenerate critical points whose $\mu({{a}})$ critical values are placed equiangularly on a circle centered at the origin. The next paragraph gives a more conceptual explanation of this symmetric arrangement.

There exists a diagonal action of $\mathbb{Z}/\mu({{a}})\mathbb{Z}$ on $\mathbb{C}^n$ which makes $z_1+p_{{a}}$ a Lefschetz fibration with a cyclic symmetry of order $\mu({{a}})$.  This diagonal action is given by \begin{align*}z_k\mapsto \eta^{(-1)^{k-1}\mu(f_{k-1}{{a}})}z_k,
\end{align*}where $\eta$ is the $\mu({{a}})$th root of unity with the smallest positive argument.

We will also consider the Lefschetz fibration \begin{align}f_{{a}}:\{p_{{a}}(z_1,\ldots ,z_n)=1\}\to \mathbb{C}\end{align} given by projection to $z_1$. This map is a a Lefschetz fibration with a cyclic symmetry of order $d({{a}})$ with the diagonal action: \begin{align}z_k\mapsto \zeta^{(-1)^{k-1}d(f_{k-1}{{a}})}z_k.
\end{align}

We now show the relationship between the Lefschetz fibrations $p_{{a}}+ z_1:\mathbb{C}^n\to \mathbb{C}$ and $f_{(1,{a})}: \{\epsilon z_1+p_{{a}}(z_1,\ldots ,z_n)=1\}\to \mathbb{C}$. Note that in the second fibration we called the coordinates $\epsilon, z_1,\ldots ,z_n$.

We now take $d({{a}})$-fold and $\mu({{a}})$-fold branched covers of $p_{{a}}$ and $f_{(1,{a})}$ respectively and obtain the Lefschetz fibrations \begin{align}
\{z_1+p_{{a}}(z_1,\ldots ,z_n)=t^{d({{a}})}\}\to \mathbb{C},
\end{align} with the map being projection to $t$ coordinate, and
\begin{align}
 \{s^{\mu({{a}})} z_1+p_{{a}}(z_1,\ldots ,z_n)=1\}\to \mathbb{C},
\end{align} with the map being projection to $s$ coordinate.

If we restrict the base of these two fibrations to $\mathbb{C}^*$, they become isomorphic. More precisely, consider the map $ \mathbb{C}^*\to \mathbb{C}^*$ that sends $s\mapsto t=\frac{1}{s}$. We can find a map of the total space that cover this map of bases \begin{align}
z_k\mapsto t^{\mu(l_{n-k}{ a})d(f_{k-1}a)}z_k.
\end{align} Notice that this directly comes from the $\mathbb{C}^*$-action.

As a result we obtain a Lefschetz fibration over $\mathbb{P}^1$ with ${\mu({{a}})}\cdot {d({{a}})}$ singularities (see Figure \ref{p1}). Let us denote this Lefschetz fibration by $\pi_a: X_{{{a}}}\to \mathbb{P}^1$, where we have homogeneous coordinates $[s:t]$ on the base.  Note that this recovers the definition given in the introduction (Equation \ref{eqpi}).

\begin{remark}Note that there is an action of $\mathbb{Z}/d({{a}})\mu({{a}})\mathbb{Z}$ on $X_{{{a}}}$, which lifts the action of $\mathbb{Z}/d({{a}})\mu({{a}})\mathbb{Z}$ on $\mathbb{P}^1$ that multiplies $s$ with the $d({{a}})\mu({{a}})$th root of unity. The $\mathbb{Z}/\mu({{a}})\mathbb{Z}$ and $\mathbb{Z}/d({{a}})\mathbb{Z}$ actions described above are both induced from this action.\end{remark}
 
 \begin{figure}
\includegraphics[width=\textwidth]{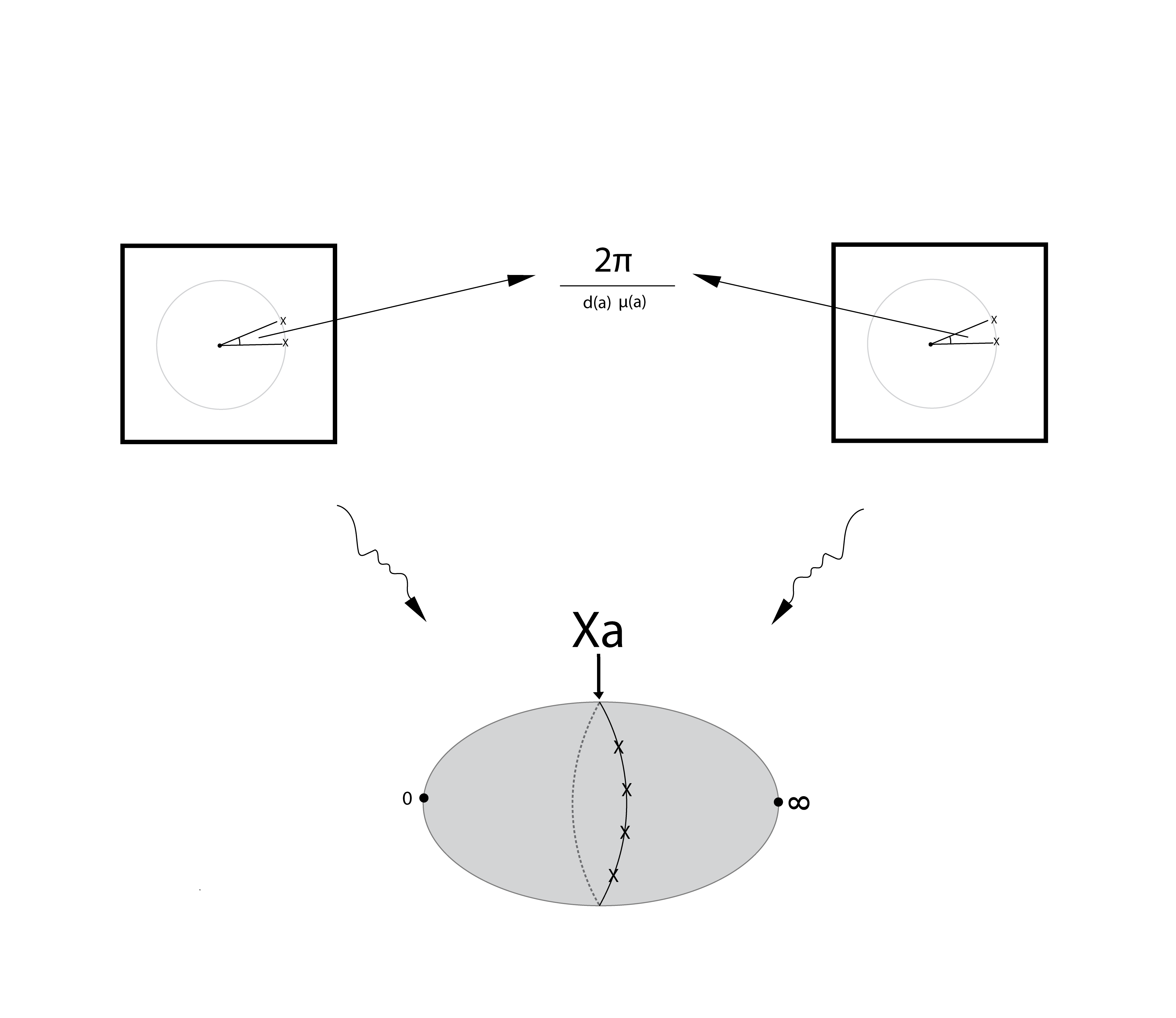}
\caption{$z_1+p_a$ and $f_a$ glue to a LF over $\mathbb{P}^1$ after passing to branch covers}
\label{p1}
\end{figure}

\begin{lemma}\label{lemmatame}
The map $\pi_a: X_a\to \mathbb{P}^1$ is tame.
\end{lemma}
\begin{proof}
Consider the compactification of $X_a$ inside $\mathbb{P}(1,q_1,\ldots,q_n,1)$ and use Section \ref{sspencils}.
\end{proof}

We now note a strange looking lemma. This will be used to justify the next section's use in computing the Seifert form.

\begin{lemma}\label{lemmatamefull}
Consider the variety $$\{\epsilon s^{\mu(a)}x_1+p_a(x_1,\ldots ,x_n)=t^{a_1\ldots a_n}\}\subset(\mathbb{P}(1,q_1,\ldots,q_n,1)-\{s=t=0\})\times \mathbb{C}_{\epsilon}$$ with its map to $\mathbb{P}^1\times \mathbb{C}$ given by $([s:t],\epsilon)$.
\begin{enumerate}
    \item This is a tame map with respect to an Ehresmann connection which is equivariant with respect to the $\mathbb{Z}/d({{a}})\mu({{a}})\mathbb{Z}$-action.
    \item There is a path from $([0:1],1)$ to $([1:1],0)$ such that the parallel transport map acts as identity on the corresponding fibers with respect to the canonical identifications of each with $M_a$.
\end{enumerate}
\end{lemma}

\begin{proof}
 The first part is a straightforward generalization of Lemma \ref{lemmatame}. For the second part, take the concatenation of the path $([0:1],T)$ from $T=1$ to $T=0$ with the path $([T:1],0)$ from $T=0$ to $T=1$. It is easy to see that this path does the job as each fiber on the path is also canonically identified with $M_a$. 
\end{proof}

\subsection{Vanishing paths}\label{ss3.3}
Consider $z_1+p_{a}$ and let us consider $A\in\mathbb{C}$ with a large absolute value, which lies in one of the rays from the origin to the critical values, and the paths drawn as in the left side of Figure \ref{basis} to be as our basis of vanishing paths. Let us call this basis of vanishing paths (more accurately, its homotopy class) \textbf{spiraling}. We want to compute the Seifert matrix $Seif(a)$ with respect to the corresponding distinguished basis in $H_{n}(\mathbb{C}^n,(p_{a}+z_1)^{-1}(A))$. The orientations of the vanishing cycles will be specified momentarily.

\begin{figure}
\includegraphics[width=\textwidth]{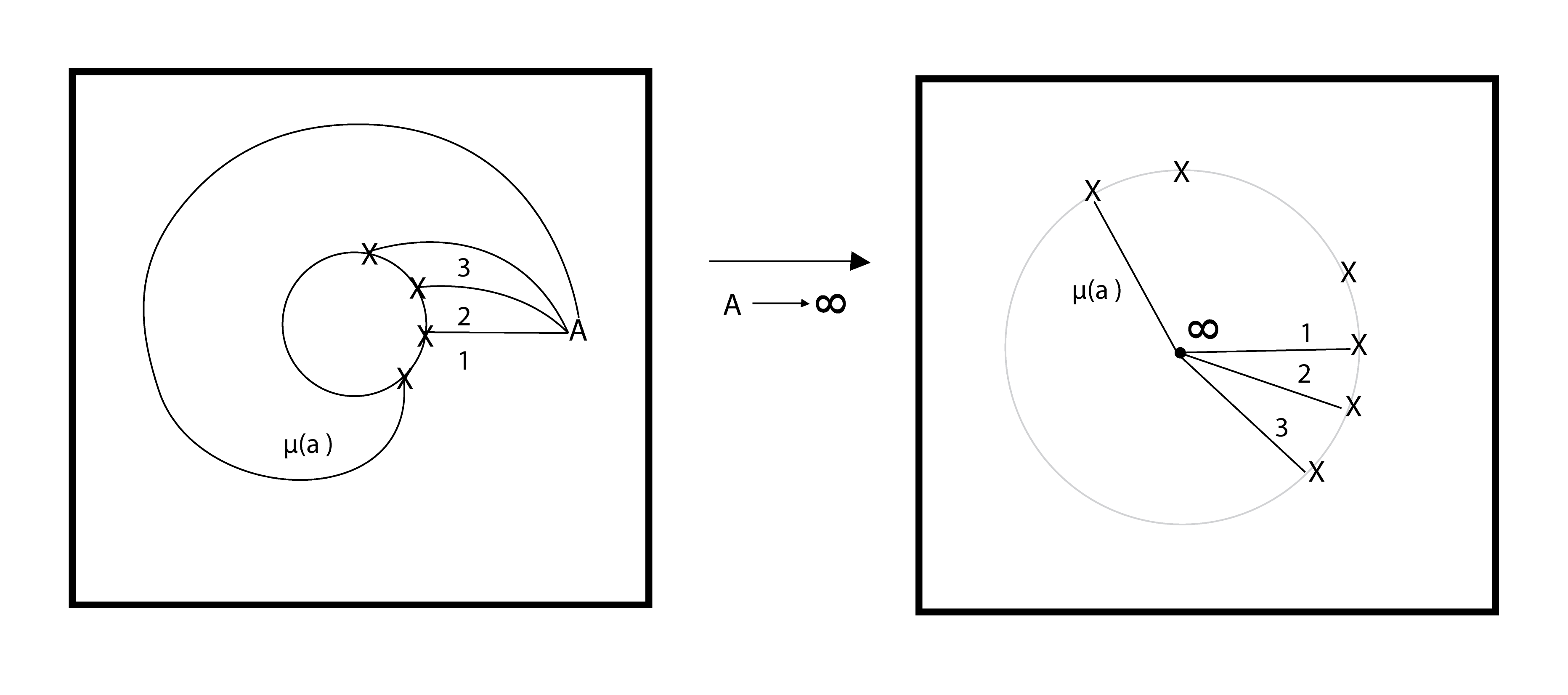}
\caption{The basis of vanishing cycles}
\label{basis}
\end{figure}

%
We can do the computation of the intersection numbers of the vanishing cycles in the other side of $X_a$. Namely, we consider $f_{(1,{a})}: \{\epsilon z_1+z_1^{a_1}z_2+\ldots +z_{n-1}^{a_{n-1}}z_{n}+z_n^{a_n}=1\}\to \mathbb{C}$, and take our base point at the origin with a radial branch cut. We consider $\mu(a)$-many adjacent radial paths as our system of vanishing paths. See the right hand side of Figure \ref{basis}. The orientations that we alluded to in the previous paragraph come from the orientation convention we introduced for Lefschetz fibrations with cyclic symmetries in Section \ref{ss2.4}.

To see the equivalence of the two Seifert matrices, we fix a preimage of $A$ and lift the spiraling paths to the base of $X_a\to \mathbb{P}^1$, and move the regular point to the point at infinity dragging the vanishing paths. Since $\mu<d$, we can then consider these paths in $f_{(1,{a})}: \{\epsilon z_1+z_1^{a_1}z_2+\ldots +z_{n-1}^{a_{n-1}}z_{n}+z_n^{a_n}=1\}\to \mathbb{C}$, if we want to.

It is not automatic that the basis we obtain of $H_{n-1}(M_a)$ this way is distinguished. This is only true because of the second part of Lemma \ref{lemmatamefull}.

\subsection{Analysis of the Lefschetz fibrations $f_{\tilde{a}}$}\label{ss3.4}
Let $\tilde{a}=(a_0,\ldots,a_{n})$ and $a=(a_1,\ldots,a_{n})$, where $n\geq 1$. Also assume that $a_1>1$.

Let us start with some basic properties of \begin{align*} f_{\tilde{a}}: \{z_0^{a_0}z_1+\ldots +z_{n-1}^{a_{n-1}}z_{n}+z_n^{a_n}=1\}=\{p_{\tilde{a}}(z_0,\ldots ,z_n)=1\}\to \mathbb{C}, 
\end{align*} as we recall was given by projection to the $z_0$ coordinate.

\begin{enumerate}
\item $f_{\tilde{a}}$ is the $a_0$-fold branched cover of $f_{(1,a_1,\ldots ,a_n)}$ as in Section \ref{ss2.5}.
\item $f_{\tilde{a}}$ admits a unique cyclic symmetry of order $d(\tilde{a})$ as in Section \ref{ss2.4}.
\item $f_{\tilde{a}}^{-1}(0)=\{p_{(a_1,\ldots ,a_n)}(z_1,\ldots ,z_n)=1\}$, which is the total space of $f_{(a_1,\ldots ,a_n)}$.
\item The restriction of the cyclic symmetry of order $d(\tilde{a})$ of $f_{\tilde{a}}$ to $f_{\tilde{a}}^{-1}(0)$ is equal to the inverse of the cyclic symmetry of order $d(a_1,\ldots ,a_n)$ of $f_{(a_1,\ldots ,a_n)}$.
\end{enumerate}

Recall the definition of a petal from Definition \ref{defpetal} in Section \ref{ss1.2}. We are ready to prove the Proposition \ref{s1prop1}.

%

\begin{proof}[Proof of Proposition \ref{s1prop1}]

First of all let us introduce a slightly different LF, which is equivalent to $f_{\tilde{a}}$ that is easier to work with computationally:

$$\tilde{f}_{\tilde{a}}: \{z_0^{a_0}z_1-\ldots +(-1)^{n-1}z_{n-1}^{a_{n-1}}z_{n}+(-1)^{n}z_n^{a_n}=1\}\to \mathbb{C},$$ where we again project to $z_0$.

It is easy to construct a commutative diagram:
\begin{align}
\xymatrix{ 
\{z_0^{a_0}z_1-\ldots +(-1)^{n}z_n^{a_n}=1\}\ar[r]^{\Phi}\ar[d]_{\tilde{f}_{\tilde{a}}}& \{z_0^{a_0}z_1+\ldots +z_n^{a_n}=1\} \ar[d]^{f_{\tilde{a}}}\\ \mathbb{C}\ar[r]_{\psi}&\mathbb{C},}
\end{align} where $\psi$ is a rotation and $\Phi$ is the restriction of a diagonal automorphism of $\mathbb{C}^{n+1}$. Note that this diagram induces one for $f_{a}$ and $z_1: \tilde{f}_{\tilde{a}}^{-1}(0)\to \mathbb{C}$ as well.

A straightforward computation shows that $\tilde{f}_{\tilde{a}}$ has a critical value $\alpha$ on the positive real axis. It suffices to prove the statement of Proposition \ref{s1prop1} for $\tilde{f}_{\tilde{a}}$ in the particular case of $a_0=1$ and $\gamma$ being a vanishing path for $\alpha$ using results of Section \ref{ss2.5}.

\begin{lemma}\label{lemmabifib}
Consider the map $w_{\tilde{a}}: \{z_0^{a_0}z_1-z_1^{a_1}z_2\ldots +(-1)^{n-1}z_{n-1}^{a_{n-1}}z_{n}+(-1)^{n}z_n^{a_n}=1\}\to \mathbb{C}^2$ given by projecting to $(z_0,z_1)$. 
\begin{itemize}
    \item This is a Lefschetz bifibration for $f_{\tilde{a}}$ in the sense of Section \ref{ssbifib}.
    \item The critical point set of $w_{\tilde{a}}$ is a graph over the $z_0,z_1$ coordinate plane $\mathbb{C}^2\subset \mathbb{C}^{n+1}$ and its projection to this plane, i.e. $critv(w_{\tilde{a}})$ is the solution set of the equation \begin{align}
z_1^{a_1\ldots a_n}-c(z_0^{a_0} z_1-1)^{\mu(a_2,\ldots, a_n)}=0,
\end{align} for some real number $c>0$.
\item There are no fake critical points, therefore the critical point set of $f_{\tilde{a}}$ coincides with the simple branch points of the projection $critv(w_{\tilde{a}})\to \mathbb{C}$ (to the first coordinate).
\end{itemize}


\end{lemma}

Now we restrict to $a_0=1$ and watch the movie of critical values in the sense of Section \ref{ssbifib}. First, we observe what happens using Mathematica. Let us state our observation as a lemma, even though we do not prove it.

\begin{lemma}\label{lemmamathematica}
Let $c$ be a positive real number, and $d>\mu$ be positive integers. We let $\epsilon$ vary from $0$ to $\alpha$ in the positive real axis. The roots of the polynomial equation $z^d-c(\epsilon z-1)^{\mu}=0$ do what is shown in Figure \ref{movie}.
\begin{figure}
\includegraphics[width=\textwidth]{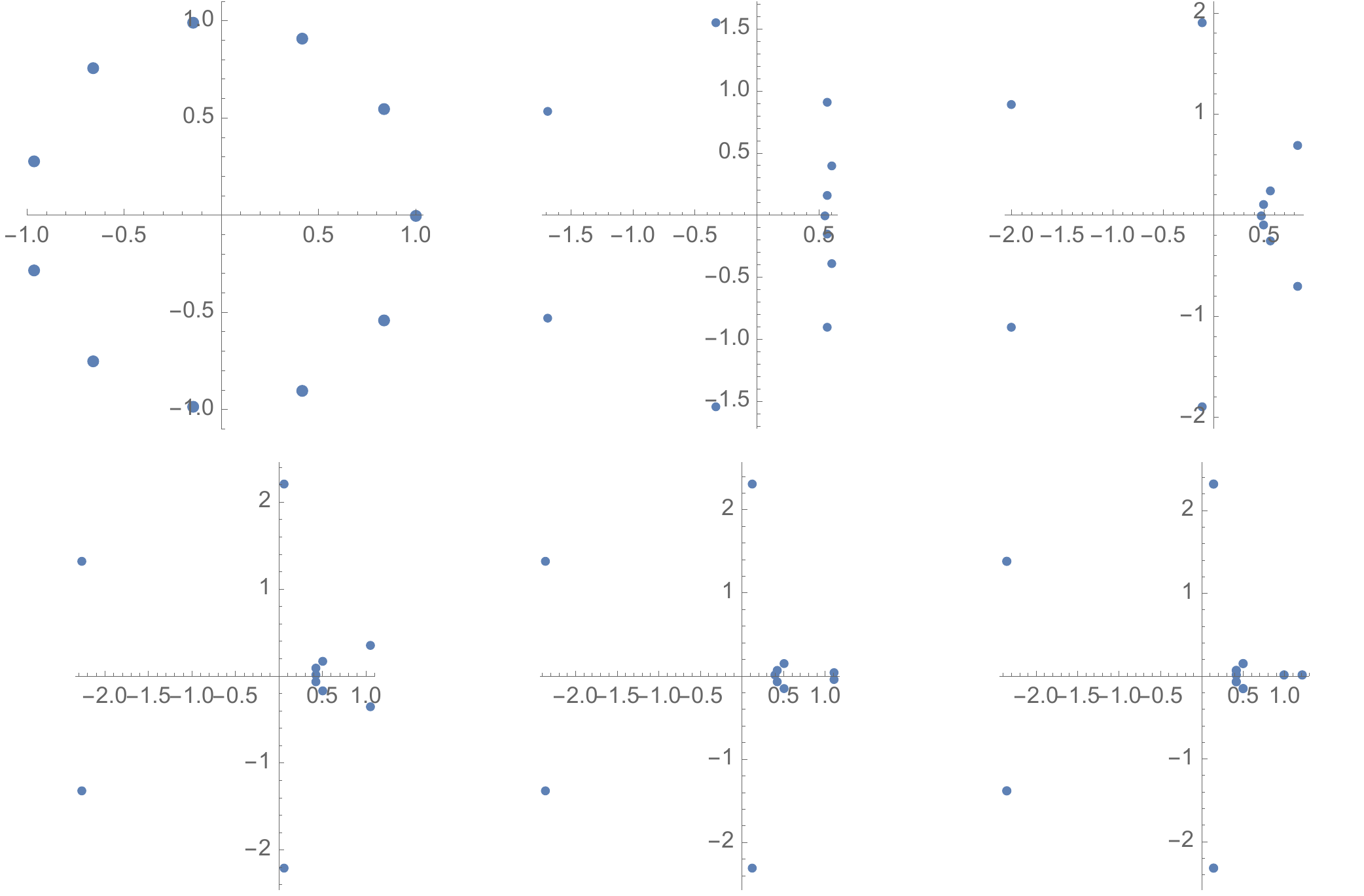}
\caption{This is the movie of critical values.}
\label{movie}
\end{figure}

\end{lemma}
If we could prove this lemma we would be done here. We will instead prove something less quantitative which is enough for our purposes. Let us first reinterpret the movie by making a change of variables. 

We can assume $c=1$ without loss of generality. Let us also define $\mu:= \mu(a_2,\ldots, a_n)$ and $d:= a_1\ldots a_n$ (also recall that $a_0=1$ and $a_1>1$). First note that $critv(w_{\tilde{a}})$ does not contain any points where $z_1=0$. Hence we can define the isomorphism of affine varieties $\{(x,z_1)\in (\mathbb{C}^*)^2\mid z_1^d=x^{\mu}\}\to critv(w_{\tilde{a}})$ by $$(x,z_1)\mapsto (\frac{x+1}{z_1},z_1).$$ Now we assume $(\mu,d)=1$ and parametrize $\{z_1^d=x^{\mu}\}$ by a $\mathbb{C}^*$ via $t\mapsto (t^d,t^{\mu})$. The general case follows from similar methods.

Notice that we have turned the problem (recalling the second and third bullet points of Lemma \ref{lemmabifib}) into analyzing roots of the trinomials $$t^d-pt^{\mu}+1,$$ which are in one-to-one correspondence with $z_0=p$ slices of $critv(w_{\tilde{a}})$, for all $p\in \mathbb{C}$. More particularly, we are interested in watching the movie of the images of the $\mu$th powers of these roots as $p$ varies in the real interval $[0,\alpha]$. This movie is precisely the same with the one from Lemma \ref{lemmamathematica}. 

\begin{remark} Geometric locations of the roots of trinomials is a heavily researched subject, which is why the translation is helpful to us. \end{remark}

First, we vary $p$ from the positive real critical point $\alpha$ to $0$, and analyze the movie in the $t$-plane. Then we will need to analyze the movie in the $t^{\mu}$-plane. Let us introduce a time variable $T$, which linearly changes $p$ from $\alpha$ to $0$, as it goes from $0$ to $1$.

At $T=0$, we have a double root on the positive real axis, and until $T=1$ no other double root occurs. The rest of the $d-2$ roots at $T=0$ are split into two groups: one with small absolute values and $\mu-1$ elements, and the other with large absolute values and $d-\mu-1$ elements (see Corollary 2.2 of \cite{dilcher}). The absolute value of the double root is between the maximum absolute value of the small absolute value group and the minimal one of the large absolute value group. Let us call these three groups small, medium, and large roots. Note that since there are no collisions for $T\in (0,1]$, the groups are always well defined. We need the following observation:

\begin{itemize}
    \item For $0<T<1$, the absolute value of a medium root is always greater than or equal to the one of a small root, and similarly for large and medium roots. 
\end{itemize}

This follows from the work of Egarvary which is summarized in \cite{trinomials}. More specifically, we use part (2) of Theorem 3.4. We push the path from $\alpha$ to $0$ so that it only intersects the real axis at the endpoints. For this path the statement of the bullet point above is true with strict inequalities. A limiting argument gives the bullet point for the original path. Note that at time $T=1$, all absolute values become the same.

Now let us go to the $t^{\mu}$-plane. It is easily checked that the collisions in the $t^{\mu}$-plane are in one to one correspondence with the ones in the $t$-plane using the definition of the trinomial along with $(\mu,d)=1$. We continue using our distinction of small, medium and large roots and call the $\mu$th powers of roots proots. We know what our matching path is for small $T$, and we want to follow through it until $T=1$. 

For every $T>0$, the circle centered at the origin that contains the two medium proots is divided into two arcs. Let us call the one that starts as the small arc for small $T$, and continously chosen for all $T$, the net at time $T$. For small $T$, this clearly gives us a matching path (the one that we are interested in). For a $T<1$, if a small or large proot intersects the net at time $T$, then they have to stay on the net until time equals $1$. At $T=1$, we know where all the proots are explicitly. We can determine which ones are the small and medium ones using the discussion surrounding Equation (6) in Section 6 of \cite{szabo}. Using the same results, we can determine which arc for $T=1$ actually is the net. It follows that all and only the small proots intersect the net, and hence the net needs to be pushed out accordingly, which gives us the desired petal.

\end{proof}


 \begin{remark}\label{longremark}
 Let us briefly go back to $z_1+p_a$ and mention another approach to understanding the geometry of the Milnor fiber of $p_a$. We consider $Y_a := (z_1+p_a)^{-1}(0)$, which is another model for the Milnor fiber. We have a tame holomorphic map $g_a: Y_a\to \mathbb{C}$, given by projecting to $z_1$, with isolated singularities. It has a degenerate critical point at the origin and $\mu(a_1,\ldots,a_n)$ non-degenerate critical points with distinct critical values equidistributed on a circle. The singularity of $g_a$ at the origin is equivalent to the singularity of $p_{(a_2,\ldots,a_n)}$ at the origin.  

 The map $g_a$ gives a handle attachment description of the Milnor fiber of $p_a$, which might be explicitly computable for $n=3$ (see the right side of Figure \ref{handle}). The first thing to do here is to understand the Milnor open book $\frac{p_{(a_1,a_2)}}{|p_{(a_1,a_2)}|}:S^3 \to S^1$. We were able to do this to some extent. 
 
 The first step is to draw a transverse knot in $S^3$ that is isotopic to the binding. We can do this in such a way that our representative admits a convex Seifert surface $S$, which can be explicitly drawn inside $S^3$. We obtain an analogue of Figure 9 in \cite{plam}, which would be the result of the same procedure for the polynomial $x^a+y^b$.
 
The Seifert surface $S$ is in fact quite special in that it can be obtained explicitly by iteratively plumbing Hopf bands. This can be proved as in Figures 4 and 5 of \cite{akbulut}. Using \cite{stallings1978constructions}, we can then explicitly compute the monodromy $\phi: S\to S$ as a product of non-separating positive Dehn twists. It also follows that the contact manifold corresponding to the open book given by $S$ with its induced area form and $\phi$ is contactomorphic to $S^3$ with its standard contact structure. This is because we know that it is diffeomorphic to $S^3$ and it is Stein fillable.

In fact more is true, namely this contactomorphism can be chosen so that it is the identity near $S$. This is a non-trivial result which follows from the Appendix of \cite{plamappendix}.

Having done all this, we would want to be able to compute the vanishing cycles coming from $g_a$ inside say $g_a^{-1}(1)$, and then actually draw them on $S$ using an explicit identification of $S$ and $g_a^{-1}(1)$ that extends to a contactomorphism $S^3\to S^3$ (this exists essenially because all fibered Seifert surfaces of the same knot are isotopic). Then, finally, we would lift them to Legendrians near $S$ as in Figure 10 of \cite{plam}, which would finish the procedure. Clearly, new ideas are needed to make this computation feasible.

 \begin{figure}
 \includegraphics[width=0.8\textwidth]{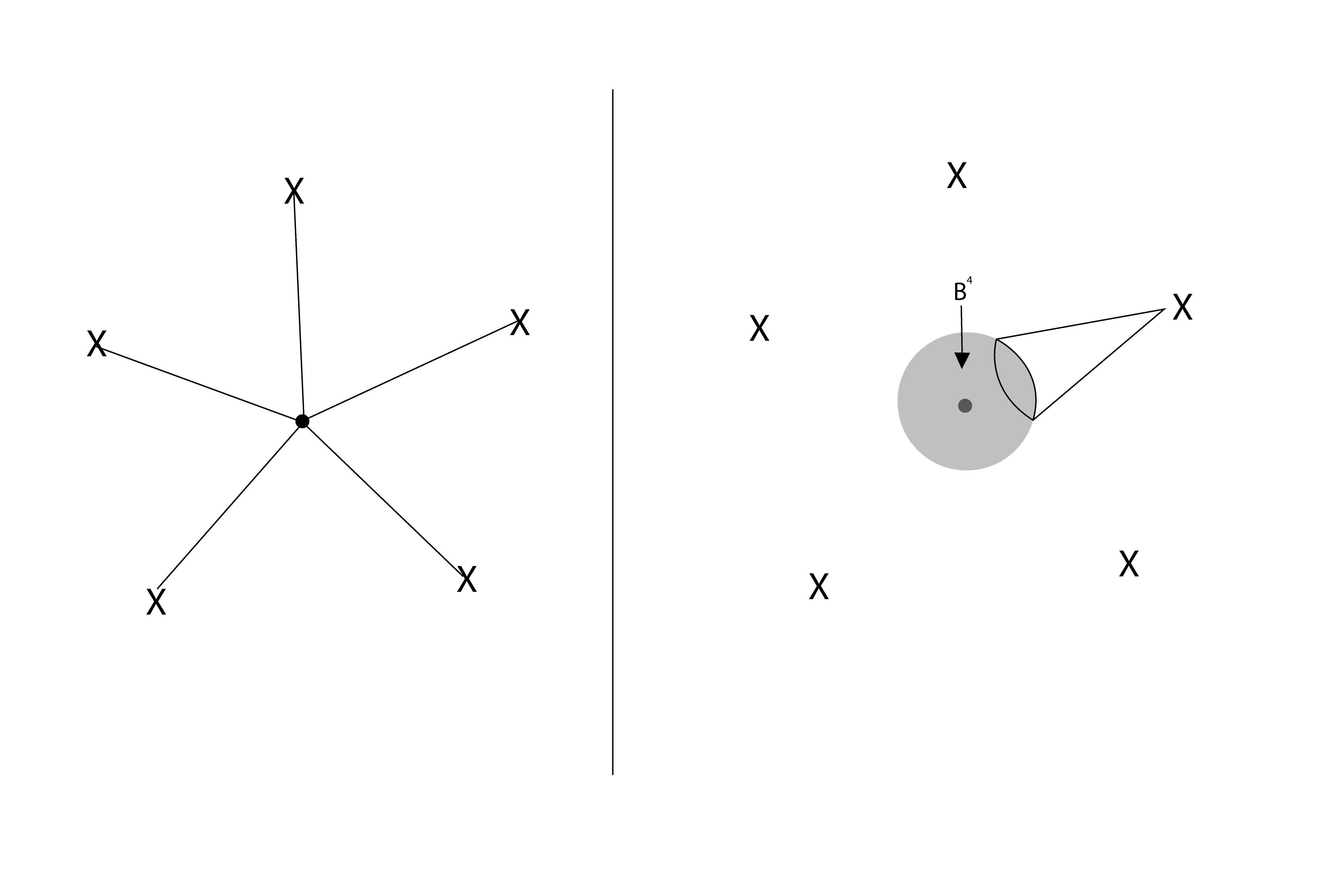}
 \caption{On the left we see the ``spokes'' as in Remark \ref{remarkkoszul}. On the right is the base of $g_a$ and a depiction of the handle attachment picture.}
 \label{handle}
 \end{figure}

 \end{remark}

 \begin{remark}\label{remarkkoszul}
 In the base of $p_a+z_1$, rather than a regular point $A$ with large absolute value and the spiraling basis of vanishing paths, we could have chosen the origin as our regular value and radial paths (with a radial branch cut) as the basis of vanishing paths. This would give us a geometric understanding of the Koszul dual distinguished basis that makes an appearance in our computation. 
 
 The vanishing cycles of radial vanishing paths in the base of $z_1+p_a$ look like "matching cycles" of a radial path from a non-degenerate critical value to the origin (i.e. the degenerate critical value) in the base of $g_a: Y_a\to \mathbb{C}$. We call these the spokes (see the left side of Figure \ref{handle}). The diagonal $\mathbb{Z}/\mu({\bf{a}})\mathbb{Z}$-action rotates the wheel.
 
 In order for this to be useful one has to perturb away the fat singular point as all the information is currently hidden there, but this breaks the symmetry and makes it difficult to see what happens afterwards.
 

 \end{remark}

\section{Orlik-Randell conjecture}\label{s4}





\begin{proof}[Proof of Theorem \ref{s1thm2}]

Recall that for $a=(a_1,\ldots , a_n)$ we had defined $M_a:=p_a^{-1}(0)$. $M_a$ is also by definition the total space of $f_a$.
\begin{itemize} 
    \item The diffeomorphism $\phi: M_a\to M_a$ is given by the cyclic symmetry of order $a_1\ldots a_n$ of $f_a$, namely the diagonal map \begin{align*}z_k\mapsto \zeta^{(-1)^{k-1}d(f_{k-1}{{a}})}z_k,
\end{align*}for $k=1,\ldots ,n$, where $\zeta=e^{\frac{2\pi i}{a_1\ldots a_n}}.$
    \item The element $v$ is given by the matching cycle of an arbitrarily chosen $\mu(a_2,\ldots a_n)$-petal in the base of $f_a$ with one of its possible orientations.
\end{itemize}

Now let us go through the numbered statements one by one.

 (1) is trivial as $\zeta$ is an $(a_1\ldots a_n)$th root of unity.

  We move on to (2). Choices can be made so that in the base of $p_a$ the parallel transport over a circular path in $\mathbb{C}$ (centered at the origin) which traces out an angle of $\theta$ is given by \begin{align*}(z_1,\ldots ,z_n)\mapsto (e^{i\theta w_1}z_1,\ldots ,e^{i\theta w_n}z_n),\end{align*} where $w_k=\frac{\mu(l_{n-k}a)}{a_k\ldots a_n}$. Therefore, a particular choice of a monodromy diffeomorphism is given by setting $\theta=2\pi$.
    
    For every $1\leq k\leq n$, we need to check that $$\frac{-\mu(a)(-1)^{k-1}d(f_{k-1}{{a}})}{a_1\ldots a_n}=\frac{\mu(l_{n-k}a)}{a_k\ldots a_n}$$ modulo $1$. This is easily seen to be true.
    
We give a separate proof of (3), even though it follows from (4), as we find the argument useful. Take a disk $D$ in $\mathbb{C}$ centered at the origin, which contains all the critical values of $p_a$, and choose a point $p$ in its boundary. We have the variation pairing on $H_n(p_a^{-1}(D),p_a^{-1}(1))$. Any diffeomorphism of $p_a^{-1}(D)$ which preserves $p_a^{-1}(1)$, and leaves the vector field generating the flow above invariant, acts on $H_n(p_a^{-1}(D),p_a^{-1}(1))$ and preserves the variation pairing. The diagonal action on $\mathbb{C}^n$ which restricts to the order $a_1\ldots a_n$ action on $M_a$ (defining $\phi$) is definitely such a diffeomorphism since $\Gamma_a$ is commutative.

(4) was proved in Sections \ref{ss3.3} and \ref{ss3.4}, but let us go through it again for clarity. In the base of $p_a+z_1:\mathbb{C}^n\to \mathbb{C}$, we have $\mu(a)$ critical points. Let us pick any one of these critical points and draw the radial ray from the origin to it. We then choose any point on this ray that is between the critical point and $\infty$ and consider the spiraling basis of vanishing paths as in Section \ref{ss3.3}.

Now, let us take the $a_1\ldots a_n$-fold branched cover of $p_a+z_1$ and compactify to obtain $X_a\to \mathbb{P}^1$ as in Section \ref{ss3.2}. We can lift the spiraling basis of vanishing paths to a system of vanishing paths on the base of the branched cover such that they all end at the same point. Then we can drag the distinguished system of paths so that they all end at the fiber at infinity, which is equal to $M_a$. Our analysis from Section \ref{ss3.4} shows that the resulting basis of vanishing cycles in $H_{n-1}(M_a)$ is given by the matching cycles of $\mu(a)$ adjacent $\mu(a_2,\ldots ,a_n)$ petals in the base of $f_a$. To finish the proof that this basis is distinguished, we use Lemma \ref{lemmatamefull}.

Choosing the initial critical point of $p_a+z_1$ differently in this argument and also using part (2) along with the fact that monodromy applied to a distinguished basis is distinguished we finish the proof. 


It will take us a little while to explain the proof of (5). The bulk of the proof will involve $f_a$ as we had mentioned before.

Assume that $n\geq 1$, and $\tilde{a}=(a_0,a_1,\ldots , a_n)$ be so that $a_0\geq 1$ and $a_i\geq 2$, for every $i\geq 1$. Also let $\gamma_1,\ldots \gamma_k$ be vanishing paths in the base of $f_{a}$ such that each $\gamma_j$ is a radial path from a critical value to the origin, and for every $1\leq j\leq k-1$, $\gamma_{i+1}$ is clockwise adjacent to $\gamma_i$.

Orient the Lefschetz thimbles and, hence, the vanishing spheres of each $\gamma_j$ (which we have computed in Proposition \ref{s1prop1} to be matching cycles of $\mu(a_2,\ldots ,a_n)$-petals in the base of $f_a$) in a way that is compatible with the symmetry of order $a_0\ldots a_n$, and call the resulting oriented vanishing cycles $\Delta_i$. We define the $k\times k$-matrix \[
    (M(\tilde{a},k))_{i,j}= 
\begin{cases}
\lbrack \Delta_i \rbrack \cdot \lbrack \Delta_j \rbrack, &  i < j\\
1,  & i=j \\
0 , & \text{otherwise}
\end{cases}
\]

It is clear that $M(\tilde{a},k)$ only depends on $\tilde{a}$ and $k$, and not on the other choices involved. Let us also define $$M(a):=M((1,a_1,\ldots ,a_n),\mu(a)).$$

\begin{proposition}\label{trivprop} The Seifert matrix $Seif(a)$ of $p_a$ (as in Section \ref{ss3.3}) satisfies $$Seif(a)=M(a).$$ 
\end{proposition}
\begin{proof}
 This follows from the discussion in Section \ref{s3}.
\end{proof}
\begin{proposition} $M(\tilde{a},k)$ is a rainbow matrix for every $\tilde{a}$ and $k$ (as in Definition \ref{s1defrainbow}), i.e. 
there exist integers $\alpha_{\tilde{a},1},\ldots , \alpha_{\tilde{a},k-1}$ (the colors) such that $$M(a,k)=Id+\sum_{i=1}^{k}\alpha_{\tilde{a},i}N_k^i,$$ where $N_k$ is the regular nilpotent matrix of size $k$. Moreover, for $k\geq k'$, $M(a,k)$ is  the principal $k\times k$ minor of $M(a,k')$.
\end{proposition}
\begin{proof}
This follows from Section \ref{ss2.4}.
\end{proof}

\begin{proposition}Assume that $k\leq \mu(a)$, \begin{enumerate}
\item $\alpha_{\tilde{a},j}=0$, for $j> \mu(a_2,\ldots a_n)$.
\item $\alpha_{\tilde{a},\mu(a_2,\ldots a_n)}=(-1)^{n}$.
\end{enumerate}
\end{proposition}
\begin{proof}
\begin{enumerate}
\item This immediately follows from Proposition \ref{s1prop1}.

\item This number is equal to the intersection number of the oriented matching cycles of two $\mu(a_2,\ldots,a_n)$-petals as in Figure \ref{sum2}.
\end{enumerate}
\begin{figure}
\includegraphics[width=0.8\textwidth]{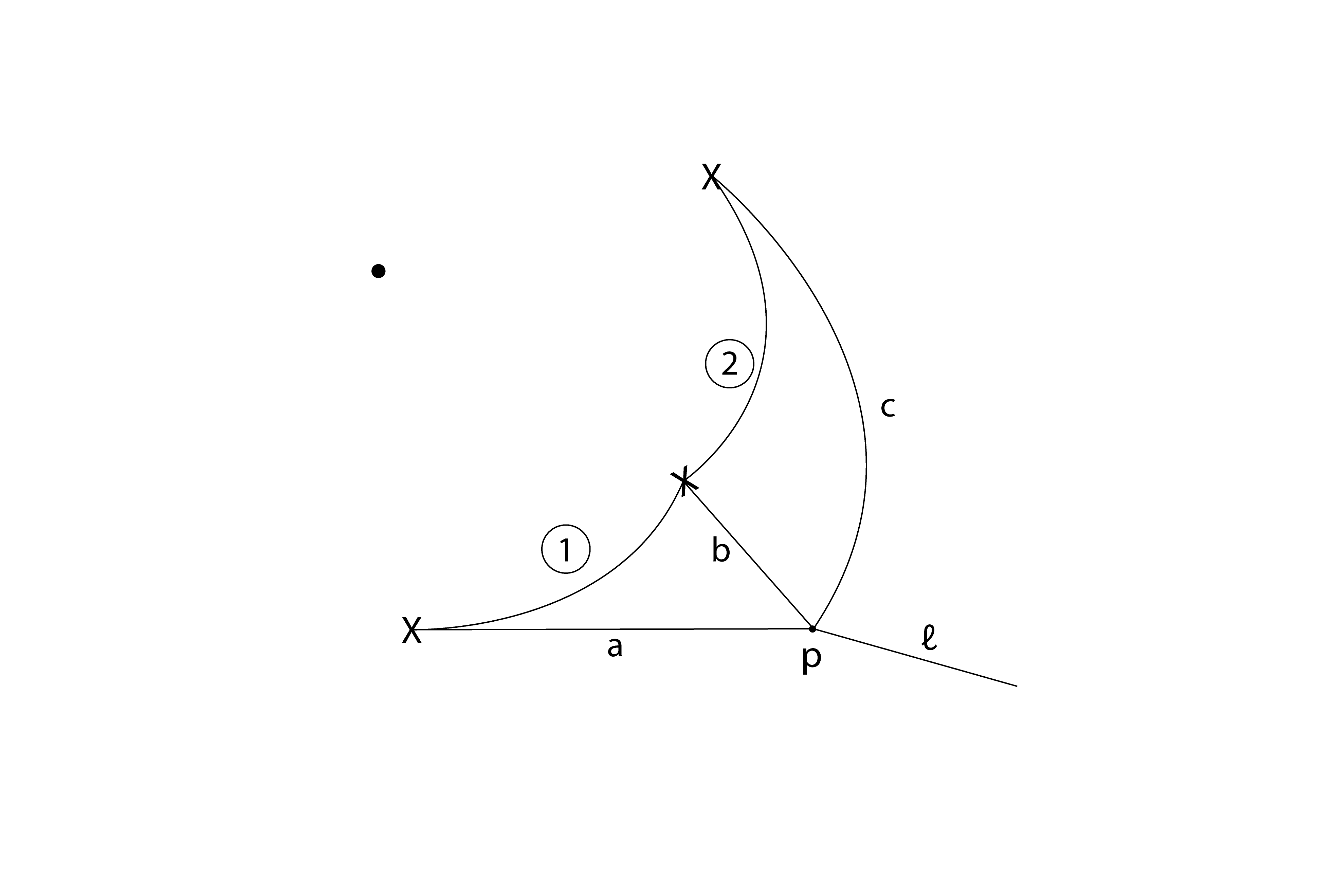}
\caption{Computing intersection numbers of petals I}
\label{sum2}
\end{figure}

Using Section \ref{ss2.7}, we have 
\begin{align*}
    \large\textcircled{\small{1}} \cdot \large\textcircled{\small{2}} & = (a-b) \cdot_l (b-c) \\
    & = a \cdot_l b - b \cdot _l b \\
    & = 1 - (-1)^{n-1}-1 = (-1)^{n}\\ 
\end{align*}

\end{proof}

\begin{proposition}
We have the following inductive formula:
$$M(\tilde{a},\mu(l_{n-2}{\tilde{a}}))= (M(a,\mu(l_{n-2}{a})))^{-1}.$$
\end{proposition}
\begin{proof}
 This time our petals intersect as in Figure \ref{sum1}.
 
\begin{figure}
\includegraphics[width=0.8\textwidth]{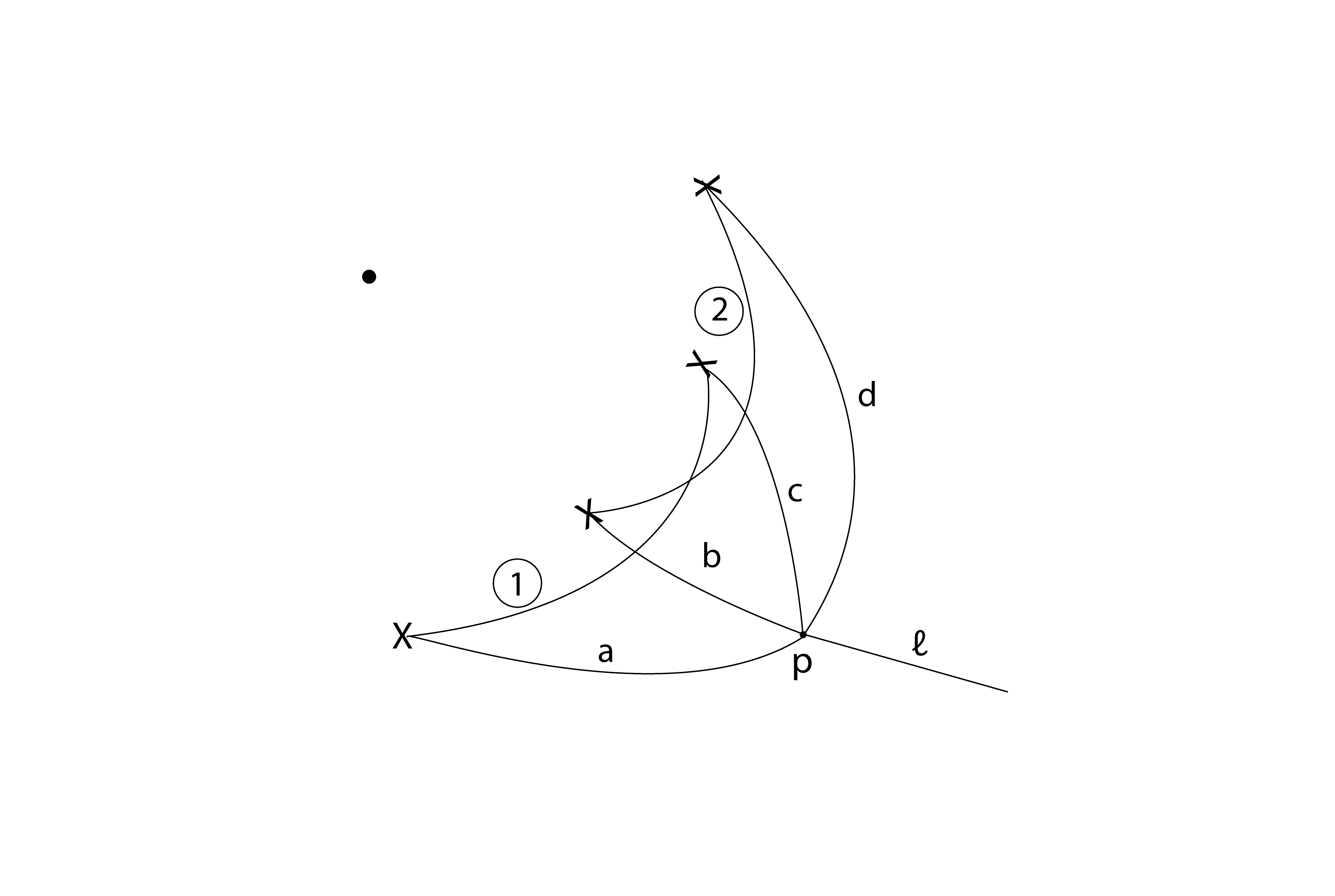}
\caption{Computing intersection numbers of petals II}
\label{sum1}
\end{figure}
 
Hence, we have:
 \begin{align*}
    \large\textcircled{\small{1}} \cdot\large\textcircled{\small{2}} & = (a-c) \cdot_l (b-d) \\
    & = a \cdot_l b - (a-c) \cdot _l d \\
    & = a \cdot_l b\\ 
\end{align*}
The result now follows from Section \ref{ss2.6} up to checking that orientations that we get from Lemma \ref{s2lemma} obey our orientation convention as in Section \ref{ss2.4}. We make the following observation. If one takes the two radial paths from a critical point of $f_a$, one to the origin and one to infinity, and orients the corresponding Lefschetz thimbles so that they intersect positively, then applying the cyclic symmetry to these thimbles result in a positively intersecting pair of thimbles. Going through the variation pairing proof of Lemma \ref{s2lemma} and how the thimbles are used to make up the matching cycles of petals carefully gives the desired orientation result.
\end{proof}

Finally:

\begin{proposition}
We have $Seif(a)=S(a)$, where $S(a)$ is as defined in Section \ref{ss1.4}.
\end{proposition}
\begin{proof}
 Recall Proposition \ref{trivprop}. Therefore, it suffices to show that 
 \begin{enumerate}
     \item $M((1,a_1),a_1-1) = S(a_1)$.
     \item $M((1,a),\mu(a))$ is obtained from $M((1,a_2,\ldots,a_n),\mu(a_2,\ldots,a_n))$ the same way $S(a_1,..,a_n)$ is obtained from $S(a_2,\ldots,a_n)$.
 \end{enumerate}
 (1) is an easy computation but let us do it for completeness. We have $f_{(1,a_1)}:\{z_0z_1+z_1^{a_1}=1\}\to \mathbb{C}$ projecting to $z_0$. $f_{(1,a_1)}$ has $a_1$ critical values and we take the $a_1-1$ adjacent radial paths. Their vanishing cycles in $\{z_1^{a_1}=1\}$ are given by the pairs of points with adjacent arguments. More precisely, consider the $a_1$ points of $\{z_1^{a_1}=1\}$ in the $z_1$ plane. The first oriented vanishing cycle is one of the adjacent pair of points with one of them assigned - and the other +. To obtain the next one we take the counter-clockwise $2\pi/n$ rotation of this cycle (recall the orientation convention). Since we have $a_1-1$ vanishing paths, we get exactly the matrix $S(a_1).$

 (2) follows from Propositions 4.2-4,  along with Section \ref{ss2.5} to relate $f(a_1,\ldots,a_n)$ with $f(1,a_2,\ldots,a_n)$.
\end{proof}

Lastly, we come to the proof of (5). It is clear that we can write: $$\phi^{-1}_*v=-\alpha_1''v\ldots -\alpha_{\mu(a)}''\phi^{\mu(a)-1}_*v$$ for some integers $\alpha_1'',\ldots ,\alpha_{\mu(a)}''$.

Note that $\phi^{-1}_*v$ is equal to monodromy operator applied to $\phi^{\mu(a)-1}_*v$ by part (2). It is well-known that the monodromy is equal to $$(-1)^{n}S(a)^{-1}S(a)^T$$ with respect to the basis $v,\ldots ,\phi^{\mu(a)-1}_*v$ using part (4). Notice that \begin{equation}
S(a)^{-1}=\begin{bmatrix} 
1 & \alpha_1' & \alpha_2' & \alpha_3' & \ldots & \alpha_{\mu(a)-1}' \\
0 & 1  & \alpha_1' & \alpha_2' &\ldots & \alpha_{\mu(a)-2}'\\
0 & 0 & 1 &\alpha_1' &\ldots&\alpha_{\mu(a)-3}'\\
 &&&\ldots&&&\\
0 &0&0&\ldots&1&\alpha_1'\\
0 &0&0&0&\ldots&1\\
\end{bmatrix}.
\end{equation}

By direct computation we obtain $$-\alpha_i''=(-1)^n\alpha_{\mu(a)-i}',$$ for $i=1,\ldots \mu(a)$, where $\alpha_0'=1$. Finally we notice that $$\alpha_{\mu(a)-i}'=(-1)^{n-1}\alpha_{i}'$$ directly from the definition. This gives the desired result.

\end{proof}
\appendix

\section{Critical point computations}\label{appa}
The following recursion will be the key point of these computations.

\begin{lemma}
Let $a_0,a_1,\ldots$ be real numbers. Consider the recursion $$\xi_{k+2}=a_k\xi_k-(a_{k+1}-1)\xi_{k+1},$$
which starts from initial data of real numbers $\xi_0$ and $\xi_1$.

If $\xi_k=0$ for some $k\in\mathbb{Z}_{\geq 0}$, then $$\xi_{k+l}=(-1)^{l-1}\xi_{k+1}\mu(a_{k+1},\ldots ,a_{k+l-1}),$$ for every $l\geq 2$.
\end{lemma}

\begin{proof}
 Straightforward computation.
\end{proof}

Let us also note the following useful identity once again:
$$\mu(a_1,\ldots ,a_n)+\mu(a_2,\ldots ,a_n)=a_1\ldots a_n.$$

\subsection{Critical points of $p_a+z_1:\mathbb{C}^n\to\mathbb{C}$}

Critical points $(z_1,\ldots ,z_n)$ of this map satisfy \begin{align}\label{crit1}
1+a_1z_1^{a_1-1}z_2=0, z_1^{a_1}+a_2z_2^{a_2-1}z_3=0,\ldots ,z_{n-1}^{a_{n-1}}+a_nz_n^{a_n-1}=0.
\end{align}

Observe that none of the $z_i'$s can be zero. Let us run the above recursion with initial data $\xi_0=0$ and $\xi_1=1$ for $0,a_1,\ldots, a_n$. We obtain the integers $\xi_0,\xi_1,\ldots, \xi_{n+1}$. It is easily seen that for every $k\in\{0,1,\ldots,n+1\}$, there exists $c_k\in\mathbb{R}$ depending only on $a_1,\ldots, a_n$ such that $$z_k=c_k\cdot z_1^{\xi_k},$$ where $z_{n+1}$ is defined as $1$. In particular, we obtain $$1=c_{n+1}z_1^{\mu(a_1,\ldots,a_{n})}$$ using the lemma. The  $c_k$'s can easily be computed, but their explicit value is not important here.

To compute the critical values, note that the equations \ref{crit1} are equivalent to:
\begin{align*}
z_1=-a_1z_1^{a_1}z_2=\ldots =(-1)^{n-2}a_1\ldots a_{n-1}z_{n-1}^{a_{n-1}}z_n=(-1)^{n-1}a_1\ldots a_nz_n^{a_n}.
\end{align*}

Hence, we obtain:  \begin{align*} p_a(z_1,\ldots ,z_n)+z_1=\frac{\mu(a_1,\ldots ,a_n)}{a_1\ldots a_n}z_1.\end{align*}

\subsection{Critical points of $f_a:M_a\to\mathbb{C}$}
Critical points $(z_1,\ldots ,z_n)$ of this map satisfy \begin{align}
z_1^{a_1}+a_2z_2^{a_2-1}z_3=0,\ldots ,z_{n-1}^{a_{n-1}}+a_nz_n^{a_n-1}z_{n+1}=0,
\end{align}where we define $z_{n+1}=1$ and of course $p_a(z_1,\ldots,z_n)=1$.

Observe that none of the $z_i'$s can be zero. Now we will consider the ansatz $$z_k=c_k\cdot z_1^{\alpha_k}\cdot z_2^{\beta_k}.$$ We again can compute these exponents by using the recursion, one that starts with $\xi_0=1, \xi_1=0$  and the other with $\xi_0=0, \xi_1=1$ for $a_1,\ldots,a_n.$ Omitting the details, we obtain $$1=c_{n+1}^{\pm 1}z_1^{a_1\cdot\mu(a_3,\ldots, a_n)}z_2^{-\mu(a_2,\ldots, a_n)}.$$

We can also easily obtain $$cz_1^{a_1}z_2=1$$ using the same strategy with the second step in the previous section. Putting the two equations together, we obtain $$1=Cz_1^{a_1(\mu(a_3,\ldots, a_n)+\mu(a_2,\ldots, a_n))}=Cz_1^{a_1\ldots a_n}.$$

\subsection{Critical set of $w_{\tilde{a}}$}
Critical points $(z_0,\ldots ,z_n)$ satisfy \begin{align}
z_1^{a_1}-a_2z_2^{a_2-1}z_3=0,\ldots ,z_{n-1}^{a_{n-1}}-a_nz_n^{a_n-1}z_{n+1}=0,
\end{align}where we define $z_{n+1}=1$ and $$z_0^{a_0}z_1-\ldots +(-1)^{n-1}z_{n-1}^{a_{n-1}}z_{n}+(-1)^{n}z_n^{a_n}=1.$$

Observe that none of the $z_1,\ldots z_n $ can be zero. Now we consider the ansatz $$z_k=c_k\cdot z_1^{\alpha_k}\cdot z_2^{\beta_k},$$ and obtain
$$1=c_{n+1}^{\pm 1}z_1^{a_1\cdot\mu(a_3,\ldots, a_n)}z_2^{-\mu(a_2,\ldots, a_n)},$$ where now we know that the coefficient is a positive real number.

It is also easy to obtain $$z_2=c\frac{1-z_0^{a_0}z_1}{z_1^{a_1}}$$ using the second step as above, where again $c$ is positive. We now plug this equation in the previous one and obtain: 
$$z_1^{a_1\ldots a_n}-C(z_0^{a_0}z_1-1)^{\mu(a_2,\ldots,a_n)}=0.$$

One can easily check that this curve is smooth. Finally, note that the branch points of the projection of this curve in $\mathbb{C}^2$ to the $z_0$ coordinate satisfy
$$a_1\ldots a_nz_1^{a_1\ldots a_n-1}-C\mu(a_2,\ldots,a_n)z_0^{a_0}(z_0^{a_0}z_1-1)^{\mu(a_2,\ldots,a_n)-1}=0,$$ which implies $z_0^{a_0}z_1=C'$, and leads to $$z_0^{a_0\ldots a_n}=C''.$$ These exactly correspond to the critical points of $f_{\tilde{a}}$, as desired.

%
%
%
%
\section{Mathematica code}\label{appd}
Below is the Mathematica code used to discover Proposition \ref{s1prop1}.

\begin{figure}[h]
\includegraphics[width=\textwidth]{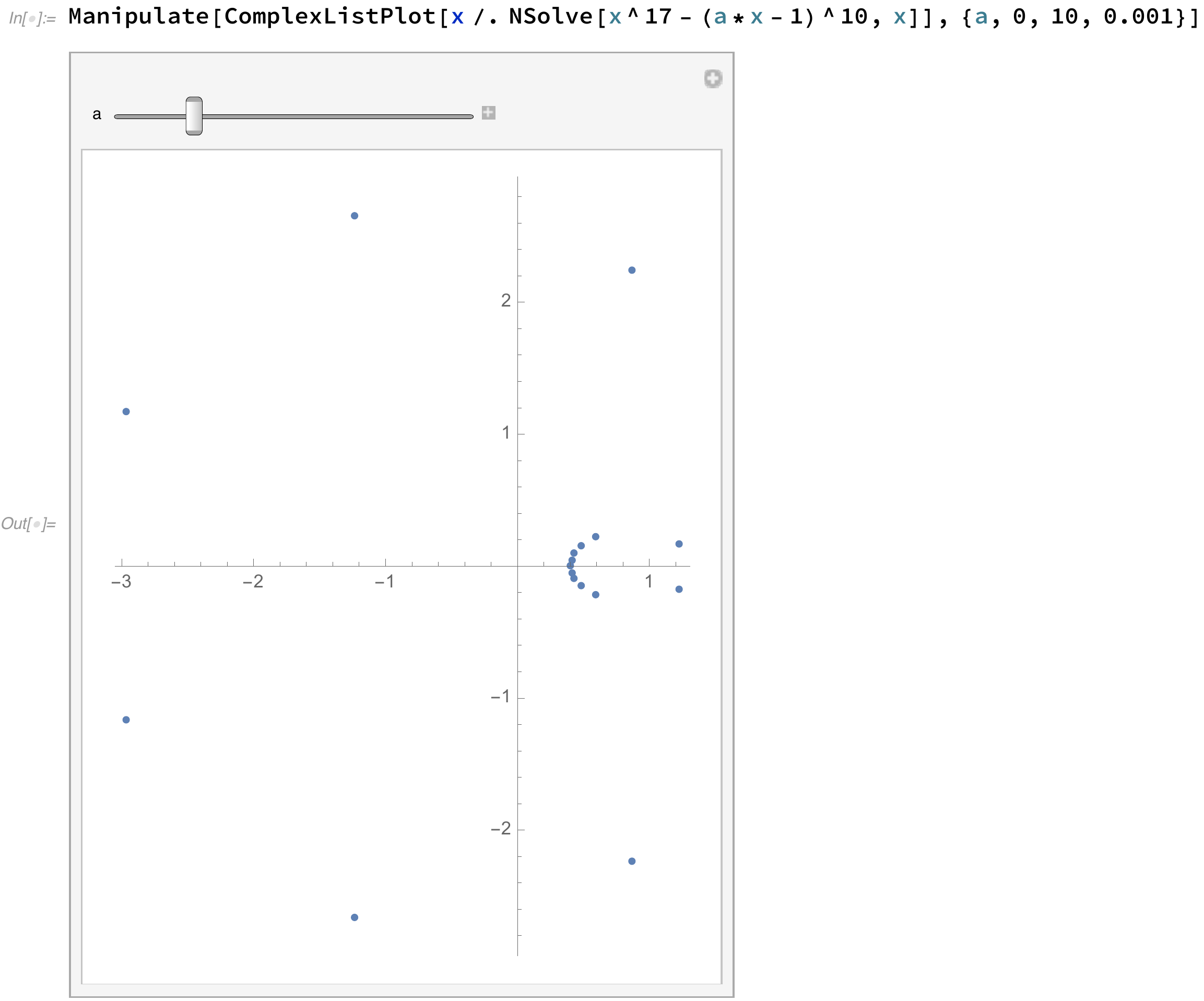}
\label{code}
\end{figure}

\section{Variation pairing}

Let us consider a smooth map of connected smooth manifolds with boundary $g: M^{2n}\to B$, where $C:=\partial B$ is an oriented circle and $M-\partial{M}$ is oriented. Note that $\partial M=g^{-1}(C)$. We assume that $g$ is a submersion on $g^{-1}(N)$, where $N$ is a collar neighborhood of $C$. Moreover, we will assume that we are given an Ehresmann connection (i.e. a choice of a horizontal subbundle) on $g^{-1}(N)\to N$ such that the unique lift to $g^{-1}(N)$ of any compactly supported vector field on $N$, which is tangent to $C$, is complete. This is a weakened version of the tameness we introduced in Section \ref{sstame}.

Let us fix a point $c\in C$, and let $F:=g^{-1}(c)$. Let us also define $P:=\partial M-F$ for convenience later. Our goal is to define a bilinear pairing, called the variation pairing, on the relative homology group $$H_n(M,F,\mathbb{Z}).$$ 
\begin{remark}We are basically recalling the construction in Section 1.1 of \cite{svar} (resulting in the definition given in Equation (1.4)), with the only notable difference that we are removing some of the unnecessary assumptions. The set-up there fits squarely within the one presented here by removing the horizontal boundary of the map in Setup 1.1, which does not affect anything. 
\end{remark}
First of all, note that the canonical map $$H_n(M-P,F,\mathbb{Z})\to H_n(M,F,\mathbb{Z})$$ is an isomorphism, as the inclusion $M-P\to M$ is a homotopy equivalence. To prove the latter statement, one finds a non-negative smooth function $\rho$ on $\partial M$ that vanishes precisely along $F$ using Whitney extension theorem, constructs the vector field $\tau(r)\rho(x)\frac{\partial}{\partial r}$ for $(x,r)\in \partial M\times [0,1)$ in a collar neighborhood of $\partial M$, where $\tau$ is a cut-off function, and uses its time $1$ map to construct a homotopy inverse. 

Let $V, W$ be two classes in $H_n(M,F,\mathbb{Z})$. We want to define their variation pairing $$V\cdot_g W\in \mathbb{Z}.$$ Assume that they correspond to classes $V',W'$ in $H_n(M-P,F,\mathbb{Z}).$

Let $\tilde{c}\in C$ be different than $c$, and define $\tilde{F}:=g^{-1}(\tilde{c})$ and $\tilde{P}:=\partial M-\tilde{F}$. By classical algebraic topology, there is an intersection pairing (for example see Definition 13.1 and Definition 13.18, in particular Equation (13.20), of Dold \cite{dold}, which also works with the same sign convention as we do from Section \ref{ss2.1}):
$$\cdot: H_n(M-\tilde{P},\tilde{F},\mathbb{Z})\times H_n(M-P,F,\mathbb{Z})\to H_0(M-\partial M)\cong \mathbb{Z}.$$Here we can imagine $M$ as sitting inside the manifold obtained by adding an open collar to its boundary to fit into the set-up of Dold squarely.

Finally, we construct a diffeomorphism $\phi: M\to M$, which is the identity outside a neighborhood of $\partial M$ properly contained in $g^{-1}(N)$, and sends $F$ to $\tilde{F}$. We do this by constructing a compactly supported vector field on $N$ (tangent to $C$) whose time $1$ map sends $c$ to $\tilde{c}$ in the positive direction, taking its lift using the connection and defining $\phi$ to be the time $1$ map. In particular, we obtain a map $$\phi_*:  H_n(M-P,F,\mathbb{Z})\to H_n(M-\tilde{P},\tilde{F},\mathbb{Z}).$$

We define \begin{align}\label{eqvarpair}V\cdot_g W:=\phi_*V' \cdot W'.\end{align} This finishes the construction of the variation pairing. The reader will notice that only the restriction of the map $g$ to a neighborhood of $\partial M$ plays any role in the construction. 

\begin{remark}A more general construction could be given for $M$ a manifold with boundary, $Z\subset \partial M$ a closed subset, and $\phi_t$, $t\in [0,1]$, a smooth isotopy of $\partial M$, which displaces $Z$ from itself.  

Consider the case $M=B^4$ is the unit ball, $Z\subset S^{3}=\partial B^4$ is a compact smooth surface with boundary $L$, and $\phi_t$ is supported near $Z$ and pushes $Z$ in the normal direction. Then, the construction results in the Seifert pairing on $H_1(Z,\mathbb{Z})\cong H_2(B^4,Z,\mathbb{Z})$ of $L$ with respect to its Seifert surface $Z$ commonly used in knot theory (see Section 2 of \cite{kauf}, which also contains a generalization to higher dimensional knot theory). Generally speaking, the Seifert pairing is highly dependent on $Z$, but important knot invariants can be gleaned from it nevertheless.  

When $L$ is fibred, choosing $Z$ to be a fiber Seifert surface (unique up to isotopy!) we get a knot invariant of $L$ by considering the Seifert pairing (now unimodular!) up to the action of $GL(H_1(Z,\mathbb{Z}))$ on bilinear forms. To bring these ideas to a full circle, we refer the reader to the Theorem 4.1 of \cite{durfee}, which exploits the special structures (more specifically, the existence of distinguished bases) of variation pairings associated to Lefschetz fibrations over the disk.\end{remark}

Let us note that there is not a geometrically meaningful non-trivial intersection pairing $$H_n(X,\tilde{Z},\mathbb{Z})\times H_n(X,Z,\mathbb{Z})\to  \mathbb{Z},$$ where $X^{2n}$ is a manifold,  and $Z$ and $\tilde{Z}$ are submanifolds, even if $Z$ and $\tilde{Z}$ are disjoint. This can be seen by taking $X$ to be the plane and $Z$ and $Z'$ to be disjoint pairs of points. Hence we had to be slightly careful above.
\bibliographystyle{plain}
\bibliography{chain}

%

\end{document}